\documentclass[11pt,reqno]{amsart}
\usepackage[utf8]{inputenc}
\usepackage[top=35mm, bottom=35mm,left=35mm, right=35mm]{geometry}
\usepackage[english]{babel}
\usepackage{amssymb}
\usepackage{amsmath}
\usepackage{amsthm}
\usepackage{amstext}
\usepackage{hyperref}
\usepackage{enumerate}
\usepackage{indentfirst}
\usepackage{xcolor}
\usepackage{cleveref}
\usepackage{mathtools}

\usepackage{graphicx}
\usepackage[T1]{fontenc}
\usepackage{lmodern}
\usepackage{float}
\usepackage{pdfpages}
\usepackage{blindtext}
\numberwithin{equation}{section}
\usepackage{extarrows}
\usepackage{tabu}
\usepackage{array}
\usepackage[utf8]{inputenc}
\usepackage{indentfirst}
\usepackage{color}
\usepackage{tikz-cd}
\usepackage{soul}
\usepackage{fancyhdr}
\usepackage{mathdots}
\allowdisplaybreaks

\usepackage{cleveref}% Automatische Benennungen

\usepackage{setspace}
\crefname{section}{Section}{Sections}
\crefname{equation}{Equation}{Equation}
\crefname{lemma}{Lemma}{Lemmata}
\crefname{bem}{Remark}{Remarks}
\crefname{kor}{Corollary}{Corollaries}
\crefname{defin}{Definition}{Definitions}
\crefname{prop}{Proposition}{Propositions}
\crefname{folg}{Conclusion}{Conclusions}
\crefname{bsp}{Example}{Examples}
\crefname{claim}{Claim}{Claims}
\crefname{figure}{Figure}{Figures}
\crefname{section}{Section}{Sections}
\crefname{enumi}{Punkt}{Punkte}
\crefname{enumii}{Unterpunkt}{Unterpunkte}
\crefname{equation}{Equation}{Equations}
\crefname{satz}{Theorem}{Theorems}
\crefname{theorem}{Theorem}{Theorems}
\crefname{nota}{Notation}{Notations}
\crefname{con}{Construction}{Constructions}

\newtheorem{thm}{Proposition}[section]

\newtheorem{lem}[thm]{Lemma}
\newtheorem{kor}[thm]{Corollary}

\newtheorem*{sat*}{Theorem}
\newtheorem{sat}[thm]{Theorem}

\theoremstyle{definition}
\newtheorem{defi}[thm]{Definition}
\newtheorem{bei}[thm]{Example}
\newtheorem{bem}[thm]{Remark}

\newtheorem{con}[thm]{Construction}

\newcommand{\N}[0]{{\mathbb N}}

\newcommand{\C}[0]{{\mathbb C}}

\newcommand{\Q}[0]{{\mathbb Q}}

\newcommand{\Z}[0]{{\mathbb Z}}
\newcommand{\F}[0]{{\mathbb F}}
\newcommand{\PZ}[0]{{\mathbb P}}
\newcommand{\A}[0]{{\mathbb A}}
\newcommand{\La}[0]{{\mathbb L}}

\renewcommand{\Im}[0]{\operatorname{Im}}

\newcolumntype{C}[1]{>{\centering\arraybackslash}m{#1}}
\AtBeginDocument{}

\DeclareMathOperator{\Ker}{Ker}

\DeclareMathOperator{\Id}{Id}

\DeclareMathOperator{\End}{End}
\DeclareMathOperator{\GL}{GL}

\DeclareMathOperator{\SL}{SL}
\DeclareMathOperator{\PSL}{PSL}
\DeclareMathOperator{\Gr}{Gr}
\DeclareMathOperator{\Spec}{Spec}

\DeclareMathOperator{\CH}{CH}

\DeclareMathOperator{\Dim}{Dim}
\DeclareMathOperator{\Sect}{Sect}
\DeclareMathOperator{\FGL}{FGL}
\DeclareMathOperator{\di}{div}
\DeclareMathOperator{\gr}{gr}
\DeclareMathOperator{\Sym}{Sym}

\DeclareMathOperator{\IG}{IG}
\DeclareMathOperator{\Sp}{Sp}
\DeclareMathOperator{\Sm}{Sm}
\DeclareMathOperator{\Sch}{Sch}

\DeclareMathOperator{\diag}{diag}

\title{Equivariant cobordism of smooth projective spherical varieties}

\begin{document}
	
	\author[Henry July]{Henry July}
	\address{Bergische Universit\"at Wuppertal, Fakult\"at 4, Gau\ss stra\ss e 20, Wuppertal, Germany}
	\email{\href{mailto:hjuly@uni-wuppertal.de}{hjuly@uni-wuppertal.de}}
	\thanks{The author is supported by the DFG Research Training Group 2240: \textit{Algebro-Geometric Methods in Algebra, Arithmetic and Topology.}}
	\begin{abstract}
		We study the equivariant cobordism rings for the action of a torus $T$ on smooth varieties over an algebraically closed field of characteristic zero. We prove a theorem describing the rational $T$-equivariant cobordism rings of smooth projective $G$-spherical varieties with the action of a maximal torus $T$ of $G$. As an application, we obtain explicit presentations for the rational equivariant cobordism rings of smooth projective horospherical varieties of Picard number one.  
	\end{abstract}
	\maketitle
	\vspace{-0.61em}
	
	\section{Introduction}
	Let $k$ be an algebraically closed field of characteristic zero and $G$ a connected reductive group over $k$. The algebraic equivariant cobordism groups were originally introduced for smooth schemes by Deshpande \cite{Deshpande}. This theory was developed independently for all $k$-schemes by Krishna as well as Heller and Malagón-López in \cite{CobofSch,HellerMalagon-Lopez} and is based on the similar construction of equivariant Chow groups presented by Totaro \cite{TotaroChowClassifying} and Edidin-Graham \cite{EdidinGraham}. Many properties of equivariant cobordism were proved in \cite{CobofSch,HellerMalagon-Lopez} building on the theory of non-equivariant cobordism developed by Levine and Morel \cite{LM}. In \cite{CobTorus}, the theory of equivariant cobordism was presented with a special focus in the case where the underlying group is a torus. 
	
	At present, there are already some computations known for this cohomology theory. The equivariant cobordism was computed for toric varieties and flag bundles in \cite{KrishnaUma,Krishna_FlagBundles}. Furthermore, the localisation formula for rational equivariant Chow groups was proved by Brion \cite{BrionTorusActions} and then extended by Krishna \cite{CobTorus} to rational equivariant cobordism which was used in order to describe the rational equivariant cobordism rings of flag varieties and symmetric varieties in \cite{Krishna_Kiritchenko}. The aim of this paper is to study further classes of examples for which the rational equivariant cobordism rings can be computed. As a consequence, one obtains a presentation of the rational ordinary cobordism rings using \cite[Theorem 3.4]{CobTorus}. Now, we describe some of our main results. 
	
	In this paper, all schemes are assumed to be quasi-projective $k$-schemes and all group actions to be linear. We are mainly interested in rational $T$-equivariant cobordism rings of smooth projective spherical varieties with an action of a torus $T$. Brion obtained the first presentation of the rational equivariant Chow rings of smooth projective spherical varieties in \cite[Theorem 7.3]{BrionTorusActions} using the equivariant intersection theory of Edidin and Graham \cite{EdidinGraham} which was more recently generalised to equivariant $K$-theory by Banerjee and Can in \cite[Theorem 1.1]{K-Theory}. Building mainly on Brion's methods we describe the rational equivariant cobordism rings of smooth projective spherical varieties (cf. Theorem \ref{EquivCobThmSpherical}) after a short recollection of some of the main known results and notions which are essential for the computations. The method of localisation was already used in \cite[Theorem 7.8]{CobTorus} in order to describe the rational equivariant cobordism rings for smooth filtrable (e.g. smooth projective spherical) schemes with finitely many $T$-fixed points and only finitely many $T$-stable curves. The aim of this article is to generalise this result for the class of smooth projective spherical varieties with possibly infinitely many $T$-stable curves. Among others, this requires an explicit computation of the equivariant cobordism rings of the projective plane and the Hirzebruch surfaces coming from the pullback maps $i^\ast:\Omega_T^\ast(X^{T'})_\Q\to \Omega_T^\ast(X^T)_\Q$ for all singular codimension one subtori $T'$ in $T$. The proof is based on the following result (cf. Theorem \ref{Relation_In_Equiv_Cobordism_FGL}) describing a relation in equivariant cobordism.
	\begin{sat*}
		Let $X$ be a smooth $T$-variety, $[h:Y\to X]$ the equivariant fundamental class of a $T$-stable cobordism cycle and $f\in k(Y)$ a rational $T$-eigenfunction with weight $\chi$ where $Z_0$ and $Z_\infty$ are the zeros and poles of $f$. Furthermore, we assume that $Z_0$ and $Z_\infty$ are smooth, i.e. that the corresponding sections are transverse. Then the relation
		\begin{align*}
			c_1^T(L_{\chi})\cdot [Y\to X]=h_\ast F_\La\left([Z_0\to Y],[-1]_{F_\La}[Z_\infty\to Y]\right)
		\end{align*}
		holds in $\Omega_\ast^T(X)$ where $F_\La$ denotes the universal formal group law and $[-1]_{F_\La}$ is the inverse in the universal formal group law. 
	\end{sat*}
	
	Using the computations of equivariant cobordism for the projective plane and the Hirzebruch surfaces $\F_n$, we can formulate the main result (cf. Theorem \ref{EquivCobThmSpherical}) by applying the technique of localisation where $\rho_{n/m}$ is an operator on $\Omega^\ast_T(k)_\Q$ (see Definition \ref{Reminder_Quotient}). 
	\begin{sat*}
		For any smooth projective and spherical $G$-variety $X$, the pullback map 
		\begin{align*}
			i^\ast:\Omega^\ast_T(X)_\Q\to \Omega^\ast_T(X^T)_\Q
		\end{align*}
		is injective. Moreover, the image of $i^\ast$ consists of all families $(f_x)_{x\in X^T}$ such that
		\begin{enumerate}[(i)]
			\item $f_x\equiv f_y\mod c_1^T(L_\chi)$ whenever $x$ and $y$ are connected by a $T$-stable curve with weight $\chi$.
			\item $(f_x-f_y)+\rho_{1/2}c_1^T(L_{\alpha})(f_z-f_x)\equiv 0\mod c_1^T(L_{\alpha})^2$ whenever $\alpha$ is a positive root of $G$ relative to $T$, $x,y$ and $z$ lie in a connected component of $X^{\Ker(\alpha)^0}$ isomorphic to a projective plane $\PZ^2$ and $x\geq y\geq z$ are ordered by their corresponding weights.
			\item $f_w-f_x-f_y+f_z\equiv 0\mod c_1^T(L_{\alpha})^2$ whenever $\alpha$ is a positive root of $G$ relative to $T$, $w,x,y$ and $z$ lie in a connected component of $X^{\Ker(\alpha)^0}$ isomorphic to $\F_0$ and $w\geq x, y \geq z$ are ordered by their corresponding weights. 
			\item $\rho_{n/2}c_1^T(L_{\alpha})(f_y-f_z)+\rho_{-n/2}c_1^T(L_{\alpha})(f_w-f_x)\equiv 0\mod c_1^T(L_{\alpha})^2$ whenever $\alpha$ is a positive root of $G$ relative to $T$, $w,x,y$ and $z$ lie in a connected component of $X^{\Ker(\alpha)^0}$ isomorphic to a rational ruled surface $\F_n$, $n\geq 1$, and $w\geq x\geq y \geq z$ are ordered by their corresponding weights. 
		\end{enumerate}
	\end{sat*} 
	As an application of Theorem \ref{EquivCobThmSpherical}, we compute the equivariant cobordism rings of horospherical varieties of Picard number one which were classified by Pasquier \cite{Pasquier_horospherical_classification} and very recently studied in \cite{Perrin_Geometry_of_horospherical_varieties}. One particular example is the class of odd symplectic Grassmannians $\IG(m,2n+1)$ for integers $n\geq 2$ and $m\in [2,n]$ which were widely studied in the past for example in \cite{Pech_Quantum_cohom_symplectic_Gr,Perrin_Geometry_of_horospherical_varieties}. These computations are done by describing very precisely the geometry of the relevant varieties where we observe in particular that the geometric and algebraic approach for the computation of the equivariant cobordism ring of $\IG(m,2n+1)$ coincide. Furthermore, we give an algorithm describing the geometry and therefore also the equivariant cobordism rings of all horospherical varieties of Picard number one. 
	
	Lastly, we recall the notion of equivariant multiplicities (cf. Definition \ref{equivariant_mulitplicities}) at nondegenerate fixed points $x\in X$ (cf. Definition \ref{non_degenerate_fixed_points}) from \cite[Section 4]{BrionTorusActions} , i.e. the tangent space $T_xX$ contains no nonzero fixed point. This will be used in order to generalise the known results for equivariant Chow rings to equivariant cobordism for smooth projective $T$-varieties $X$ (cf. Proposition \ref{Non-smooth_classes}). To be more precise, we determine the classes $[f:Y\to X]$ for smooth varieties $Y\subseteq X$ in which all fixed points are nondegenerate using equivariant multiplicities (cf. Lemma \ref{Smooth_classes}). In addition, we compute the classes $[f:Y\to X]$ for smooth $Y$ assuming that all fixed points in $X$ and all fixed points in the fibers $f^{-1}(x)$  are nondegenerate for each $x\in X^T$ (cf. Proposition \ref{Non-smooth_classes}). Furthermore, using the previous results we give the explicit example of the odd symplectic Grassmannian $\IG(2,5)$ in which the classes are computed. Finally, we observe that different resolutions of singularities of singular varieties $X_m\subseteq X$ coming from the filtration (\ref{T-filtration}) of smooth projective $T$-varieties $X$ lead to different classes in the equivariant cobordism ring of $\IG(2,5)$ as opposed to the equivariant Chow rings.  
	  
	\vspace*{3mm}
	\textit{Acknowledgements.} I am deeply indebted to Nicolas Perrin for providing key ideas behind the present results which were clarified in very helpful discussions. I would like to thank Michel Brion for his great explanations of some of his results which were used in this article. I am also grateful to Jens Hornbostel for his valuable comments and suggestions during various stages of this work. Additionally, I would like to thank Herman Rohrbach, Christoph Spenke and Thomas Hudson for many helpful discussions. 
	
	\section{A relation in equivariant cobordism}
	In this section, we start with the basic definitions and properties of algebraic cobordism before defining equivariant cobordism. For more details on the properties of algebraic cobordism and equivariant cobordism we refer the reader to the book of Levine and Morel \cite{LM} and the articles of Krishna \cite{CobTorus,CobofSch}, respectively. Before we can define algebraic cobordism, we recall the definition of a formal group law and the construction of the Lazard ring $\La$ after introducing the main notations. 
	
	\subsection{Notations}Let $k$ be an algebraically closed field of characteristic zero and $G$ a connected reductive linear algebraic group over $k$. We denote the category of quasi-projective schemes over $k$ by $\boldsymbol{\Sch}_k$ and the full subcategory consisting of smooth and quasi-projective schemes over $k$ by $\boldsymbol{\Sm}_k$. A scheme is meant to be an object of $\boldsymbol{\Sch}_k$. Similarly, if $G$ is a linear algebraic group over $k$, we denote the category of quasi-projective schemes over $k$ with a $G$-action and $G$-equivariant maps by $G-\boldsymbol{\Sch}_k$. Frequently these schemes will be called $G$-schemes. The corresponding category of smooth and quasi-projective $G$-schemes will be denoted by $G-\boldsymbol{\Sm}_k$. We assume all group actions to be linear, i.e. for any $G$-action on a scheme $X$ there exists a representation $G\to \GL(V)$ on a finite-dimensional $V$ such that $X\to \PZ(V)$ is a $G$-equivariant immersion. This assumption is always fulfilled for normal schemes which was proved by Sumihiro in \cite{Sumihiro_Equiv_Completion_II}. Furthermore, we assume all representations of $G$ to be finite-dimensional. Lastly, throughout this article we will use the notion of $T$-stable subsets whereas our main sources (e.g. \cite{BrionTorusActions,CobTorus}) use the term $T$-invariant subsets for the same property.
	\begin{defi}
		A \textbf{commutative formal group law of rank one} with coefficients in $R$ is a pair $(R,F_R)$ consisting of a commutative ring $R$ and a formal power series $F_R(u,v)=\sum a_{ij}u^iv^j\in R[[u,v]]$ satisfying the following conditions.
		\begin{enumerate}[(i)]
			\item $F(u,0)=F(0,u)=u\in R[[u]]$.
			\item $F(u,v)=F(v,u)\in R[[u,v]]$.
			\item $F(u,F(v,w))=F(F(u,v),w)\in R[[u,v,w]]$.			
		\end{enumerate}
	\end{defi}
	The Lazard ring is a polynomial ring over $\Z$ which is generated by infinitely but countably many variables. It is constructed as the quotient of the polynomial ring $\Z[\{A_{ij}\vert (i,j)\in \N^2\}]$ by the relations obtained by imposing the conditions of a commutative formal group law on the $A_{ij}$. This uniquely defines the universal commutative formal group law $F_{\La}$ of rank one on $\La$ which is given by 
	\begin{align*}
		F_{\La}(u,v)=\sum_{i,j}a_{ij}u^iv^j\in \La[[u,v]]
	\end{align*}
	where $a_{ij}$ is the equivalence class of $A_{ij}$ in $\La$. The grading in the Lazard ring is given by assigning the degree $i+j-1$ to the coefficient $a_{ij}$. The resulting graded ring will be denoted by $\La_\ast$. Alternatively, we could assign degree $1-i-j$ to the coefficient $a_{ij}$ in which case we denote the resulting commutative graded ring by $\La^\ast$. Furthermore, the graded formal power series ring will be denoted by $\La[[u_1,...,u_n]]_{\gr}$ and its equivalent with rational coeffcients is given by the graded topological tensor product $(\La[[u_1,...,u_n]]_{\gr})_\Q:=\La[[u_1,...,u_n]]_{\gr}\widehat{\otimes}_\Z\Q$ which was described in more detail in \cite{CobofSch}.
	
	Recall the existence of a unique formal graded power series $\chi(u_i)\in \La[[u_1,...,u_n]]_{\gr}$ which satisfies $F_{\La}(u_i,\chi(u_i))=0$. For any positive integer $b\in \Z_{\geq 1}$ and $[0]_{F_\La}u_i:=0$ we establish the following notations.
	\begin{align*}
		u_i+_{F_\La}u_j&:=F_\La(u_i,u_j)\in \La[[u_i,u_j]]_{\gr},\\
		[-1]_{F_\La}u_i&:=\chi(u_i)\in \La[[u_i]]_{\gr},\\
		u_i-_{F_\La}u_j&:=F_\La(u_i,\chi(u_j))\in \La[[u_i,u_j]]_{\gr},\\
		[b]_{F_\La}u_i&:=F_{\La}(u_i,[b-1]_{F_\La}u_i)\in \La[[u_i]]_{\gr}.
	\end{align*}
	It is clear that $[b]_{F_\La}u$ is divisible by $u$ for any $u\in \La[[u_1,...,u_n]]_{\gr}$ of degree $1$.	\begin{lem}\label{FGL_rational_coefficients_power_series}
		Let $u\in \La[[u_1,...,u_n]]_{\gr}$ be a homogeneous element of degree $1$. Then there exists an element $g\in \La_\Q[[x]]$ such that $u=g([b]_{F_\La}u)$ for any $b\in \Z_{\geq 1}$. 
	\end{lem}
	\begin{proof}
		Fix $b\in \Z_{\geq 1}$ and write 
		\begin{align*}
			[b]_{F_\La}u=b_1u+b_2a_{11}u^2+b_3a_{21}u^3+b_4a_{12}u^3+b_5a_{11}^2u^3+....
		\end{align*}
		for $b_i\in \Z_{\geq 0}$ for all $i\geq 1$. Now we construct an element $\rho$ of degree $0$ such that $\rho\cdot [b]_{F_\La}u=u$ holds. By comparison of coefficients, we observe that $\rho$ is given by 
		\begin{align*}
			\rho=\frac{1}{b_1}-b_2\frac{a_{11}}{b_1^2}u+\left(-\frac{b_3}{b_1^2}a_{21}-\frac{b_4}{b_1^2}a_{12}+\left(-\frac{b_5}{b_1^2}+\frac{b_2^2}{b_1^3}\right)a_{11}^2\right)u^2+...
		\end{align*}
		Successively replacing $u$ with $\rho\cdot [b]_{F_\La}u$ implies the claim.
	\end{proof}
	The previous lemma leads to the following definition.
	\begin{defi}
		Let $u\in \La[[u_1,...,u_n]]_{\gr}$ be a homogeneous element of degree $1$. Then for $n\in \Z_{\geq 1}$ we define
		\begin{align*}
		[-n]_{F_\La}u:=[-1]_{F_\La}\left([n]_{F_\La}u\right)
		\end{align*}
		and furthermore, if there exists a homogeneous element $u'\in(\La[[u_1,...,u_n]]_{\gr})_\Q$ of degree $1$ such that $[m]_{F_\La}u'=u$ for $m\in \Z_{\geq 1}$ then we define
		\begin{align*}
			\left[\frac{1}{m}\right]_{F_\La}u:=u'.		
		\end{align*} 
	\end{defi}
	\begin{defi}\label{Reminder_Quotient}
			In the setting of the above definition we define the operator $\rho_{n/m}$ by
			\begin{align*}
				\rho_{n/m}u:=\frac{\left[n\right]_{F_\La}\left(\left[\frac{1}{m}\right]_{F_\La}u\right)}{u}
			\end{align*}
			in $(\La[[u_1,...,u_n]]_{\gr})_\Q$ for any $n\in \Z\setminus\{0\}$ and $m\in \Z_{\geq 1}$.
			
	\end{defi}
	\begin{bem}
		The quotient $\rho_{n/m}u$ is indeed in $(\La[[u_1,...,u_n]]_{\gr})_\Q$ for any $n\in \Z\setminus\{0\}$ and $m\in \Z_{\geq 1}$ because $\left[\frac{1}{m}\right]_{F_\La}u\in (\La[[u_1,...,u_n]]_{\gr})_\Q$ is homogeneous of degree $1$ and therefore, $\left[\frac{1}{m}\right]_{F_\La}u=g(u)$ holds for some $g\in \La_\Q[[x]]$ by Lemma \ref{FGL_rational_coefficients_power_series}. Further, $g(u)$ is divisible by $u$ by construction and thus, $\left[n\right]_{F_\La}\left(\left[\frac{1}{m}\right]_{F_\La}u\right)$ is divisible by $u$. 
	\end{bem}
	\subsection{Algebraic Cobordism}	
	Let $X$ be an equidimensional $k$-scheme. Then a cobordism cycle is given by a family $[f:Y\to X,L_1,...,L_r]$ where $Y$ is smooth and irreducible, the map $f$ is projective and the $L_i$ are line bundles over $Y$ whereas the number of line bundles may be empty. The degree of a cobordism cycle is given by $\dim_k(Y)-r$. Let $\mathcal{Z}_\ast$ be the free graded abelian group generated by the isomorphism classes of the cobordism cycles where the grading is given by the degree of the cycles. Now, we impose three relations on $\mathcal{Z}_\ast$ in order to define algebraic cobordism. 
	
	The first one is called the \textbf{dimension axiom}. Let $\mathcal{R}_\ast^{\Dim}(X)$ be the graded subgroup of $\mathcal{Z}_\ast$ generated by the cobordism cycles $[f:Y\to X,L_1,...,L_r]$ such that $\dim_k(Y)<r$. We denote the corresponding quotient $\mathcal{Z}_\ast(X)/\mathcal{R}_\ast^{\Dim}(X)$ by $\underline{\mathcal{Z}}_\ast(X)$. 
	   
	Secondly, for a line bundle $L$ on $X$ and a cobordism cycle $[f:Y\to X,L_1,...,L_r]$, we define the first Chern class operator on $\underline{\mathcal{Z}}_\ast(X)$ by 
	\begin{align*}
		\widetilde{c}_1(L)
		[f:Y\to X,L_1,...,L_r]=
		[f:Y\to X,L_1,...,L_r,f^\ast(L)].
	\end{align*}
	This definition is used in order to impose the \textbf{section axiom}. Given a line bundle $L$ over $X$ and a section $s:X\to L$ which is transverse to the zero section. Let $Z\to X$ be the closed zero-subscheme of $s$. Then we define $\mathcal{R}_\ast^{\Sect}(X)$ to be the graded subgroup of $\underline{\mathcal{Z}}_\ast(X)$ generated by elements of the form $\widetilde{c}_1(L)[\Id:X\to X]-[Z\to X]$. We denote the quotient $\underline{\mathcal{Z}}_\ast(X)/\mathcal{R}_\ast^{\Sect}(X)$ by $\underline{\Omega}_\ast$ which we refer to as \textbf{algebraic pre-cobordism}.
	
	Lastly, we impose the \textbf{formal group law axiom} on algebraic pre-cobordism by considering the subset $\mathcal{R}_\ast^{\FGL}(X)\subseteq \La_\ast\otimes\underline{\Omega}_\ast(X)$ consisting of elements of the form 
	\begin{align*}
		F_\La(\widetilde{c}_1(L),\widetilde{c}_1(M)([\Id:X\to X])-\widetilde{c}_1(L\otimes M)([\Id:X\to X]),
	\end{align*}
	where $L$ and $M$ are line bundles over $X$. Finally, for the subset $\La_\ast\mathcal{R}_\ast^{\FGL}(X)\subseteq \La_\ast\otimes\underline{\Omega}_\ast(X)$ which is given by elements of the form $a\otimes \rho$ for $a\in \La_\ast$ and $\rho\in \mathcal{R}_\ast^{\FGL}(X)$, we define \textbf{algebraic cobordism} of $X$ by 
	\begin{align*}
		\Omega_\ast(X)=\La_\ast\otimes\underline{\Omega}_\ast(X)/\La_\ast\mathcal{R}_\ast^{\FGL}(X).
	\end{align*}
	Let $d$ be the dimension of the equidimensional $k$-scheme $X$. In this case, we define $\Omega^i(X)=\Omega_{d-i}(X)$ for all $i\in\Z$.
	\subsection{Equivariant Cobordism}
	Recall that $G$ is a connected reductive linear algebraic group over $k$. Now we consider for any integer $j\geq 0$ a corresponding pair $(V_j,U_j)$ where $V_j$ is an $l_j$-dimensional representation of $G$ and $U_j$ is a $G$-stable open subset of $V_j$ such that the codimension of the complement $(V_j\setminus U_j)$ in $V_j$ is at least $j$. Furthermore, we ask that $G$ acts freely on $U_j$ such that the quotient $U_j/G$ is a quasi-projective scheme. Such a pair will be called a \textbf{good pair} for the $G$-action corresponding to $j$. It is well known that such a good pair always exists (cf. \cite[Lemma 9]{EdidinGraham}). 
	
	For a $k$-scheme $X$ of dimension $d$ with a $G$-action and an integer $j\geq 0$, let $(V_j,U_j)$ be an $l_j$-dimensional good pair corresponding to $j$. Then we denote the mixed quotient of the product $X\times U_j$ by the free diagonal action of $G$ by $X\times^G U_j$. 
	
	We now present one of the main results concerning actual computations of equivariant algebraic cobordism. Since the original definition is very hard to be computed in general, one can make use of the following result by Krishna \cite{CobofSch} which will serve as our definition of equivariant cobordism.
\begin{thm}\cite[Theorem 6.1]{CobofSch}\label{SequenceofGoodpairs}
	Let $\{(V_j,U_j)\}_{j\geq 0}$ be a sequence of $l_j$-dimensional good pairs such that there exist $G$-representations $(W_j)_{j\geq 0}$ with
	\begin{enumerate}[(i)]
		\item $V_{j+1}=V_j\oplus W_j$ as representations of $G$ with $\dim(W_j)>0$ and
		\item $U_j\oplus W_j\subseteq U_{j+1}$ as $G$-stable open subsets.
	\end{enumerate}
	Then for any scheme $X\in G-\boldsymbol{\Sch}_k$ of dimension $d$ and any $i\in\Z$, one has
	\begin{align*}
	\Omega_i^G(X)\overset{\cong}{\longrightarrow}\varprojlim_j \Omega_{i+l_j-g}\left(X\times^G U_j\right).
	\end{align*}
	Moreover, such a sequence of good pairs always exists. 
\end{thm} 
\begin{bem}\label{Equivariant Fundamental class}
	One should note that the equivariant algebraic cobordism can be non-zero for any $i\in\Z$ unlike the ordinary algebraic cobordism $\Omega^\ast$. Furthermore, we set 
	\begin{align*}
		\Omega_\ast^G(X):=\bigoplus_{i\in\Z}\Omega_i^G(X).
	\end{align*}
	If in addition $X$ is an equi-dimensional $k$-scheme of dimension $d$ with $G$-action, we let $\Omega^i_G(X)=\Omega_{d-i}^G(X)$ and analogously $\Omega^\ast_G(X):= \bigoplus_{i\in\Z}\Omega^i_G(X)$. We denote the equivariant cobordism $\Omega^\ast_G(k)$ of the underlying ground field by $S(G)$. Furthermore, If $G$ is the trivial group, equivariant algebraic cobordism reduces to ordinary algebraic cobordism. Besides that, equivariant algebraic cobordism with rational coefficients is again defined by the graded topological tensor product $\Omega^\ast_G(X)_\Q:=\Omega^\ast_G(X)\widehat{\otimes}_\Z\Q$ which was described in \cite{CobofSch}.
	
	For any $X\in G-\boldsymbol{\Sch}_k$ and a projective morphism $f:Y\to X$ in $G-\boldsymbol{\Sch}_k$ where $Y$ is smooth of dimension $d$ we obtain for any $j\geq 0$ and any $l_j$-dimensional good pair $(V_j,U_j)$ an ordinary cobordism cycle $[Y\times^G U_j\to X\times^G U_j]$ of dimension $d+l_j-g$ by \cite[Lemma 5.1]{CobofSch}. 
	This defines a unique element $\alpha\in \Omega^G_d(X)$ which we call the \textbf{$\boldsymbol{G}$-equivariant fundamental class} of the cobordism cycle $[f:Y\to X]$.  
\end{bem}
\begin{bem}\cite[Section 2.5]{CobTorus}\label{Chern_classes_Cobordism}
	It is well-known that $\Omega_G^\ast(X)$ is an $S(G)$-algebra if $X$ is smooth. In this case, we will identify the commutative $\La$-subalgebra of $\End_\La(\Omega^\ast_G(X))$ generated by the Chern classes of vector bundles with the $\La-$subalgebra of the equivariant cobordism ring $\Omega^\ast_G(X)$ via $c_i^G(E)\mapsto c_i^G(E)\left([\Id:X\to X]\right)$. Therefore, we will denote this image also by $c_i^G(E)$. Since we pass freely between vector bundles $E$ and their corresponding locally free coherent sheaves we will also write $c_1^G(\mathcal{E})$ for a locally free coherent sheaf $\mathcal{E}$. 
\end{bem} 
From now on, we will only consider $G$-equivariant cobordism where the group $G$ is given by some torus $T$.
\begin{thm}\label{T-equivariant_cobordism_of_the_point}\cite[Proposition 6.7]{CobofSch}
	Let $\{\chi_1,...,\chi_n\}$ be a basis of the character group of a torus $T$ of rank $n$. Then the assignment $t_i\mapsto c_1^T(L_{\chi_i})$ yields a graded $S(T)$-algebra isomorphism
	\begin{align*}
		\La[[t_1,...,t_n]]_{\gr}\cong \Omega_T^\ast(k)
	\end{align*}
	where $L_{\chi_i}$ is the $T$-equivariant line bundle over $\Spec k$ corresponding to the character $\chi_i$ of $T$.
\end{thm}
\begin{bem}
	Let $M$ be the character group of a torus $T$ of finite rank. Using Definition \ref{Reminder_Quotient} one observes that
	\begin{align*}
		\rho_{n/m}c_1^T(L_\chi)=\frac{c_1^T(L_{n\chi/m})}{c_1^T(L_\chi)}
	\end{align*}
	holds in $S(T)_\Q$ for any character $\chi\in M$, $n\in \Z\setminus\{0\}$ and $m\in \Z_{\geq 1}$ if $\frac{n\chi}{m}$ is also a character in $M$.
\end{bem}
The first step for the computations in this article is to describe a result in equivariant cobordism which is similar to the following one in Chow groups. For any $T$-scheme $X$, any closed $T$-stable subvariety $Y\subseteq X$ and any rational function $f$ on $Y$ which is an eigenvector of $T$ for weight $\chi$, we have $\chi\cdot [Y]=\di_Y(f)$ in the $\CH^\ast_T(k)$-module $\CH^\ast_T(X)$ (cf. \cite[Theorem 2.1]{BrionTorusActions}). We would like to have such a relation for smooth schemes $X$ in equivariant cobordism and therefore, we need to understand properly the $S(T)$-action on $\Omega_\ast^T(X)$ for $X\in \boldsymbol{\Sm}_k$. 
\begin{con}\label{RelationEquivCob}	
	Now we present a similar construction to the one introduced to prove the above relation in Chow groups in \cite[Theorem 2.1]{BrionTorusActions}. By Proposition \ref{T-equivariant_cobordism_of_the_point} we know that for any basis $\{\chi_1,...,\chi_n\}$ of the character group of $T$ we have the isomorphism $\La[[t_1,...,t_n]]_{\gr}\cong S(T), t_i\mapsto c_1^T(L_{\chi_i})$, where in this case we set $L_{\chi_i}$ to be the one-dimensional representation of $T$ on which $T$ acts via weight $-\chi_i$. Hereby $c_1^T(L_{\chi_i})$ means $c_1^T(L_{\chi_i})[\Spec k\to \Spec k]$ where $[\Spec k\to \Spec k]$ is by abuse of notation the equivariant fundamental class of the ordinary cobordism cycle $[\Spec k\to \Spec k]$. For any character $\chi$ and a $l_j$-dimensional good pair $(V_j,U_j)$ we have the line bundle $(L_{\chi}\times U_j)/T\to U_j/T$ which we denote by $(L_{\chi})_T$. Since equivariant cobordism is defined via an inverse limit construction we consider the elements  
	\begin{align*}
		c_1^T(L_{\chi})[\Spec k\to \Spec k]=\varprojlim_j \widetilde{c}_1((L_{\chi})_T)[U_j/T\to U_j/T].
	\end{align*}
	By \cite[Theorem 4.11]{CobTorus} we know that the $S(T)$-module $\Omega^T_\ast(X)$ is generated by the equivariant fundamental classes of the $T$-stable cobordism cycles in $\Omega_\ast(X)$ for smooth $k$-schemes $X$. Therefore, we take one of these ordinary cobordism cycles $[h:Y\to X]$ and consider $[(Y\times U_j)/T\to (X\times U_j)/T]$ in the $j$-th component of the equivariant fundamental class which we denote by $[Y\to X]_j$ for some good pair $(V_j,U_j)$.  For the morphism $g: (X\times U_j\times U_j)/T\to U_j/T$ we use the exterior product on equivariant cobordism which was described in the proof of \cite[Theorem 5.2]{CobofSch} and thus, we obtain
	\begin{align*}
		c_1^T(L_{\chi})\cdot [Y\to X]&=\varprojlim_j (\widetilde{c}_1((L_{\chi})_T)[U_j/T\to U_j/T]\cdot [(Y\times U_j)/T\to (X\times U_j)/T])\\
		&=\varprojlim_j \widetilde{c}_1(g^\ast(L_{\chi})_T)[(Y\times U_j\times U_j)/T\to (X\times U_j\times U_j)/T]
	\end{align*}
	in $\Omega_\ast^T(X)$. We observe that in this case the line bundle $g^\ast(L_\chi)_T$ is obtained by the good pair $(V_j\times V_j,U_j\times U_j)$ of dimension $2l_j$ for $j\geq 0$.
\end{con}
\begin{sat}\label{Relation_In_Equiv_Cobordism_FGL}
	Let $X$ be a smooth $T$-variety, $[h:Y\to X]$ the equivariant fundamental class of a $T$-stable cobordism cycle and $f\in k(Y)$ a rational $T$-eigenfunction with weight $\chi$ where $Z_0$ and $Z_\infty$ are the zeros and poles of $f$. Furthermore, we assume that $Z_0$ and $Z_\infty$ are smooth, i.e. that the corresponding sections are transverse. Then the relation
	\begin{align*}
		c_1^T(L_{\chi})\cdot [Y\to X]=h_\ast F_\La\left([Z_0\to Y],[-1]_{F_\La}[Z_\infty\to Y]\right)
	\end{align*}
	holds in $\Omega_\ast^T(X)$ where $F_\La$ denotes the universal formal group law and $[-1]_{F_\La}$ is the inverse in the universal formal group law. 
\end{sat}
\begin{proof} 
	We consider the rational function $f$ on $Y$ which is an eigenfunction of $T$ of weight $\chi$. One may observe that 
	\begin{align*}
		s:(Y\times U_j\times U_j)/T\to (Y\times U_j\times U_j\times L_{\chi})/T,(y,u_1,u_2)\mapsto (y,u_1,u_2,f(y))
	\end{align*}
	is a rational section of the line bundle $h^\ast g^\ast(L_{\chi})_T$. For this line bundle with the given rational section, we can also write 
	\begin{align*}
		h^\ast g^\ast(L_{\chi})_T=\mathcal{O}_{(Y\times U_j\times U_j)/T}(Z_0-Z_\infty)\cong \mathcal{O}_{(Y\times U_j\times U_j)/T}(Z_0)\otimes \mathcal{O}_{(Y\times U_j\times U_j)/T}(Z_\infty)^\vee
	\end{align*}
	by the known correspondence between Cartier divisors and pairs $(L,s)$ consisting of a line bundle and a rational section. We simplify by setting $L_{\chi^1}=\mathcal{O}_{(Y\times U_j\times U_j)/T}(Z_0)$ and $L_{\chi^2}=\mathcal{O}_{(Y\times U_j\times U_j)/T}(Z_\infty)$. By the smoothness assumption we know that the corresponding sections of $L_{\chi^1}$ and $L_{\chi^2}$ coming from the rational section $s$ are transverse to the zero sections of $L_{\chi^1}$ and $L_{\chi^2}$, respectively. Furthermore, the zero-subschemes of these sections are $T$-stable and hence they define cobordism cycles whose equivariant fundamental classes are in $\Omega_\ast^T(Y)$. In the following computation we will use \cite[Definition 2.1.2]{LM} axiom (A3) and \cite[Definition 2.2.1]{LM} axiom $(\Sect)$. We know further by \cite[Proposition 5.2.1]{LM} that the Chern class operator $\widetilde{c}_1(L)$ on a smooth scheme $X$ is given by $\widetilde{c}_1(L)(\eta)=c_1(L)\cdot \eta$ for $\eta\in \Omega^\ast(X)$ where the first Chern class is given by $c_1(L)=\widetilde{c}_1(L)(1_X)$. Lastly, we have the embeddings of the zero-subschemes $i_0: (Z_0\times U_j\times U_j)/T \to (Y\times U_j\times U_j)/T$ and similarly $i_\infty:(Z_\infty\times U_j\times U_j)/T \to (Y\times U_j\times U_j)/T$. Using all those properties, we obtain 
	\begin{align*}
		&\widetilde{c}_1(g^\ast (L_{\chi})_T)[(Y\times U_j\times U_j)/T\to (X\times U_j\times U_j)/T]\\
		&=\widetilde{c}_1(g^\ast (L_{\chi})_T)h_\ast[1_{(Y\times U_j\times U_j)/T}]\\
		&=h_\ast \widetilde{c}_1(h^\ast g^\ast (L_{\chi})_T)[1_{(Y\times U_j\times U_j)/T}]\\
		&=h_\ast \widetilde{c}_1(L_{\chi^1}\otimes L_{\chi^2}^\vee)[1_{(Y\times U_j\times U_j)/T}]\\
		&=h_\ast F_\La(\widetilde{c}_1(L_{\chi^1}),\widetilde{c}_1(L_{\chi^2}^\vee))[1_{(Y\times U_j\times U_j)/T}]\\
		&=h_\ast\left(\widetilde{c}_1(L_{\chi^1})[1_{(Y\times U_j\times U_j)/T}]+[-1]_{F_\La}\widetilde{c}_1(L_{\chi^2})[1_{(Y\times U_j\times U_j)/T}]\right)\\
		&\hspace{3.5mm} +h_\ast\left(\sum_{i,k\geq 1}a_{ik}\widetilde{c}_1(L_{\chi^1})^i\circ \widetilde{c}_1(L_{\chi^2}^\vee)^k[1_{(Y\times U_j\times U_j)/T}]\right)\\
		&=h_\ast\left(\widetilde{c}_1(L_{\chi^1})[1_{(Y\times U_j\times U_j)/T}]+[-1]_{F_\La}\widetilde{c}_1(L_{\chi^2})[1_{(Y\times U_j\times U_j)/T}]\right)\\
		&\hspace{3.5mm} +h_\ast\left(\sum_{i,k\geq 1}a_{ik}c_1(L_{\chi^1})^i\cdot c_1(L_{\chi^2}^\vee)^k\right)\\
		&=h_\ast \left({i_0}_\ast(1_{(Z_0\times U_j \times U_j)/T})+[-1]_{F_\La}{i_\infty}_\ast(1_{(Z_{\infty}\times U_j \times U_j)/T})\right)\\
		&\hspace{3.5mm} +h_\ast\left(\sum_{i,k\geq 1}a_{ik}{i_0}_\ast\left(1_{(Z_0\times U_j \times U_j)/T}\right)^i\cdot  \left([-1]_{F_\La}{i_\infty}_\ast(1_{(Z_{\infty}\times U_j \times U_j)/T})\right)^k\right)\\
		&=h_\ast\left([Z_0\to Y]_{j}+[-1]_{F_\La}[Z_\infty\to Y]_{j}+\sum_{i,k\geq 1}a_{ik}[Z_0 \to Y]_j^i\cdot \left([-1]_{F_\La}[Z_\infty\to Y]_j\right)^k\right).		    
	\end{align*} 
	Furthermore, the sum is finite since the ordinary first Chern classes are nilpotent. We conclude the claim because taking the limit on these elements commutes with the pushforward $h_\ast$ by the definition of the equivariant pushforward maps.
\end{proof}
In the sequel, we will only consider the case where $s$ is a global section which is transverse to the zero section. In this particular case, the terms containing $Z_\infty$ disappear and one obtains the following statement.
\begin{kor}\label{Relation_We_actually_use}
	Assume that $s$ is a global section which is transverse to the zero section. In this case, the relation
	\begin{align*}
		c_1^T(L_{\chi})\cdot [Y\to X]=[Z_0\to X]
	\end{align*}holds in $\Omega_\ast^T(X)$ where $Z_0$ is the zero-subscheme of $s$ on $Y$. 
\end{kor}
For equivariant Chow groups the above relations generate all relations as explained in the following result of Brion.  
\begin{sat}\cite[Theorem 2.1]{BrionTorusActions}\label{Brion 2.1}
	Let $X$ be a variety with an action of a torus $T$. The $\CH^T_\ast(k)$-module $\CH^T_\ast(X)$ is defined by generators $[Y]$, where $Y\subseteq X$ is a $T$-stable subvariety, and by relations $[\di_Y(f)]=\chi\cdot [Y]$, where $f$ is a non-constant rational function on $Y$ which is an eigenvector of $T$ of weight $\chi$. 
\end{sat}
\begin{bem}
	Similarly to Theorem \ref{Brion 2.1} we know that $\Omega_T^\ast(X)$ is generated by the equivariant fundamental classes of the $T$-stable cobordism cycles in $\Omega^\ast(X)$ by \cite[Theorem 4.11]{CobTorus} for a smooth variety $X$ with an action of a torus. At present the author does not know whether the equivariant cobordism rings $\Omega^\ast_T(X)$ are given by the equivariant fundamental classes of $T$-stable cobordism cycles in $\Omega^\ast(X)$ modulo the previously described relations from Proposition \ref{Relation_In_Equiv_Cobordism_FGL}, but it might be enough for smooth projective varieties $X$ with an action of a torus.
\end{bem}  
\subsection{Localisation at fixed points}
We now prove a lemma which will be useful in the sequel for comparing the equivariant algebraic cobordism with respect to a torus $T$ and its quotient $T/F$ by a finite subgroup $F$.
\begin{lem}\label{IsomorphismCobordism T and T/F}
	Let $T$ be a torus of rank $n$ and $F$ be a finite subgroup. Then we have a graded $\La$-algebra isomorphism
	\begin{align*}
	\Omega_T^\ast(k)_\Q\cong \Omega_{T/F}^\ast(k)_\Q.
	\end{align*}
\end{lem}
\begin{proof}
	Let $\{\chi_1,...,\chi_n\}$ be a basis of the character group of $T$. The basis of the character group of $T/F$ is then given by $\{a_1\chi_1,...,a_n\chi_n\}$ for positive integers $a_1\vert a_2\vert\cdots\vert a_n$.
	Using Proposition \ref{T-equivariant_cobordism_of_the_point} we know that there is an isomorphism $\Omega_T^\ast(k)\cong \La[[t_1,...,t_n]]_{\gr}$ mapping $c_1^T(L_{\chi_i})\mapsto t_i$ where $L_{\chi_i}$ is the one-dimensional representation of weight $-\chi_i$. Furthermore, we have $\Omega_{T/F}^\ast(k)\cong \La[[t'_1,...,t'_n]]_{\gr}$ for $c_1^{T/F}(L_{a_i\chi_i})\mapsto t'_i$. Since we consider the $L_{\chi_i}$ as the one-dimensional representations of $T$ and similarly those of $T/F$, we know that 
	\begin{align*}
	c_1^T(L_{a_i\chi_i})=c_1^T(L_{\chi_i+...+\chi_i})=c_1^T(L_{\chi_i}\otimes...\otimes L_{\chi_i})=[a_i]_{F_\La}c_1^T(L_{\chi_i})
	\end{align*}  
	holds in $\Omega_T^\ast(k)$ where $F_\La$ denotes again the universal formal group law in cobordism. On the other hand, we know that we can take $c_1^T(L_{a_i\chi_i})$ as generators of $\Omega_T^\ast(k)_\Q$ instead of $c_1^T(L_{\chi_i})$ as soon as we consider rational coefficients by Lemma \ref{FGL_rational_coefficients_power_series}. This leads to the desired isomorphism.  
\end{proof}
\begin{bem}
	The preceding lemma implies the same statement for equivariant Chow groups and furthermore we remark that the finite subgroup $F$ has order $a_1\cdots a_n$. Lastly, the statement also holds if we only take coefficients in $\Z[1/p_1,...,1/p_\ell]$ where $p_1,..,p_\ell$ are the primes occuring in the prime factorisation of $a_n$. Therefore, we only have to invert a finite number of primes in order to obtain the isomorphism of Lemma \ref{IsomorphismCobordism T and T/F}. 
\end{bem}
 To finish this introductory section, let $T$ be a torus and $X\in T-\textbf{Sch}_k$. We recap some basic notation for $T$-filtrable schemes and relevant applications which were presented by Krishna \cite{CobTorus}. We say that $X$ is \textbf{$\boldsymbol{T}$-filtrable} if the fixed point subscheme $X^T$ is smooth and projective and if there is an ordering $X^T=\coprod_{m=0}^n Z_m$ of the connected components $Z_m$ of the fixed point subscheme such that there is a filtration of $X$ by $T$-stable closed subschemes 
\begin{align}
\emptyset=X_{-1}\subsetneq X_0\subsetneq...\subsetneq X_n=X \label{T-filtration}
\end{align}
with $Z_m\subseteq W_m:=X_m\setminus X_{m-1}$ and maps $\phi_m: W_m\to Z_m$ for all $0\leq m\leq n$ which are all $T$-equivariant vector bundles such that the inclusions $Z_m\hookrightarrow W_m$ are the $0$-section embeddings. One should note that if $X$ is $T$-filtrable then so is every closed subscheme $X_m$. We remark that this definition coincides with Brion's definition in \cite[Section 3]{BrionTorusActions} for smooth projective schemes which will be the objects of our main interest. The following result which is a consequence of the Bialynicki-Birula decomposition will be essential for our understanding of the equivariant cobordism of smooth projective varieties.
\begin{thm}\cite[Theorem 4.3]{BB}\label{BB_smooth_projective}
	Let $X$ be a smooth projective variety with an action of a torus $T$. Then $X$ is $T$-filtrable.
\end{thm}
The following proposition is very useful for computing equivariant cobordism. As opposed to the localisation theorem for Chow groups (cf. \cite[Theorem 3.3]{BrionTorusActions}) one has to assume that the fixed point scheme consists only of finitely many isolated points in order to formulate the equivalent statement in equivariant cobordism.
\begin{thm}\cite[Theorem 7.6, Theorem 7.1]{CobTorus}\label{EquivCobIntersectionImages}
	Let $X$ be a smooth $T$-filtrable scheme with an action of a torus $T$. Further, let $X^T$ consist of finitely many fixed points $x_1,...,x_s$ and let $i:X^T\hookrightarrow X$ denote the inclusion of the fixed point subscheme. Then the pullback map $i^\ast:\Omega^\ast_T(X)\to \Omega^\ast_T(X^T)$ is injective and its image is the intersection of the images of
	\begin{align*}
	i^\ast_{T'}:\Omega^\ast_T(X^{T'})_\Q\to \Omega^\ast_T(X^T)_\Q 
	\end{align*} 
	where $T'$ runs over all subtori of codimension one in $T$.
\end{thm}  
Lastly, we present the equivalent result of \cite[Theorem 3.4]{BrionTorusActions} in equivariant cobordism.
\begin{thm}\cite[Theorem 7.8]{CobTorus}\label{Krishna7.8}
	Let $X$ be a smooth $T$-filtrable scheme where a torus $T$ acts with finitely many fixed points $x_1,...,x_s$ and finitely many $T$-stable curves. Then the image of 
	\begin{align*}
	i^\ast: \Omega^\ast_T(X)_\Q\to \Omega^\ast_T(X^T)_\Q
	\end{align*}
	is the subalgebra of $(f_1,...,f_s)\in S(T)^s_\Q$ such that $f_i\equiv f_j\mod c_1^T(L_\chi)$ whenever $x_i$ and $x_j$ are connected by a stable irreducible curve where $T$ acts through the weight $\chi$.
\end{thm} 
\section{Equivariant cobordism of spherical varieties}
Throughout this section let $G$ be a connected reductive group, $B\subseteq G$ a Borel subgroup and $T\subseteq B$ a maximal torus. Recall that a normal $G$-variety $X$ containing a dense $B$-orbit is called \textbf{spherical}. This section is based on \cite{BrionTorusActions}.     
\begin{defi}
	A subtorus $T'\subseteq T$ is \textbf{regular} if its centraliser 
	\begin{align*}
	C_G(T')=\{g\in G\ \vert \ gt'=t'g\text{ for all }t'\in T'\}
	\end{align*}is equal to the torus $T$. If this is not the case, we call the subtorus $T'$ \textbf{singular}.
\end{defi}
\begin{bem}\label{PositiveRootsSubtorus}
	Following \cite[Corollary B, Section 26.2]{Humphreys_Linear_alg_grps} a subtorus $T'$ of codimension one is singular if and only if it is the identity component of the kernel of some positive root $\alpha$ of $(G,T)$. In this case we will write $T'=\Ker(\alpha)^0$. Then $\alpha$ is unique and the group $C_G(T')$ is the product of $T'$ with a subgroup $S(\alpha)\subseteq G$ isomorphic to $\SL_2$ or to $\PSL_2$. Then the fixed point locus $X^{T'}$ is equipped with an action of
	\begin{align*}
	C_G(T')/T'=T'S(\alpha)/T'=S(\alpha)/(S(\alpha)\cap T')=\SL_2\text{ or }\PSL_2
	\end{align*}
	since $S(\alpha)\cap T'$ is either of order one or two. Furthermore, we have $T=T'S_m(\alpha)$ for a maximal subtorus $S_m(\alpha)$ of $S(\alpha)$, the image of the coroot of $\alpha$. As above, $T'\cap S_m(\alpha)$ is a finite group $F(\alpha)$ of order one or two. Clearly, $F(\alpha)$ acts trivially on $X^{T'}$ and hence the torus action on $X^{T'}$ is in fact the action through the corresponding quotient $T/F(\alpha)\cong (T'\times S_m(\alpha))/(F(\alpha)\times F(\alpha))$.  
\end{bem}
	In the following proposition, we will analyse the components of the fixed point subschemes $X^{T'}$ for regular and singular codimension one subtori $T'\subseteq T$. Recall that the surface $\mathbb{F}_n=\PZ(\mathcal{O}_{\PZ^1}\oplus \mathcal{O}_{\PZ^1}(n))$ is called the $n$-th Hirzebruch surface. 

\begin{thm}\cite[Proposition 7.1]{BrionTorusActions}\label{componentsX^T'}
	Let $X$ be a spherical $G$-scheme and let $T'\subseteq T$ be a subtorus of codimension one.   
	\begin{enumerate}[(i)]
		\item Each irreducible component of $X^{T'}$ is a spherical $C_G(T')$-variety.
		\item If $T'$ is regular, then the fixed point set $X^{T'}$ is at most one-dimensional.
		\item If $T'$ is singular, then $X^{T'}$ is at most two-dimensional. Furthermore, any two-dimensional connected component of $X^{T'}$ is either a Hirzebruch surface $\F_n$ where $C_G(T')$ acts through the natural action of $\SL_2$, or the projective plane $\PZ^2$ where $C_G(T')$ acts through the projectivization of a non-trivial $\SL_2$-module of dimension three. 
	\end{enumerate}
\end{thm}
In this section we want to generalise the presentations of the equivariant Chow rings of smooth projective spherical $G$-varieties (cf. \cite[Theorem 7.3]{BrionTorusActions}) to equivariant algebraic cobordism. In order to be able to generalise those, we need to compute the equivariant algebraic cobordism of the projective plane and the Hirzebruch surfaces according to Proposition \ref{componentsX^T'} and  Proposition \ref{EquivCobIntersectionImages}.
Using notation as in Proposition \ref{Krishna7.8}, we now can formulate the main result of this section which is the analogue of \cite[Theorem 7.3]{BrionTorusActions} and which will be proved later on in this section.
\begin{sat}\label{EquivCobThmSpherical}
	For any smooth projective and spherical $G$-variety $X$, the pullback map 
	\begin{align*}
	i^\ast:\Omega^\ast_T(X)_\Q\to \Omega^\ast_T(X^T)_\Q
	\end{align*}
	is injective. Moreover, the image of $i^\ast$ consists of all families $(f_x)_{x\in X^T}$ such that
	\begin{enumerate}[(i)]
		\item $f_x\equiv f_y\mod c_1^T(L_\chi)$ whenever $x$ and $y$ are connected by a $T$-stable curve where $T$ acts through the weight $\chi$.
		\item $(f_x-f_y)+\rho_{1/2}c_1^T(L_{\alpha})(f_z-f_x)\equiv 0\mod c_1^T(L_{\alpha})^2$ whenever $\alpha$ is a positive root of $G$ relative to $T$, $x,y$ and $z$ lie in a connected component of $X^{\Ker(\alpha)^0}$ isomorphic to a projective plane $\PZ^2$ and $x\geq y\geq z$ are ordered by their corresponding weights.
		\item $f_w-f_x-f_y+f_z\equiv 0\mod c_1^T(L_{\alpha})^2$ whenever $\alpha$ is a positive root of $G$ relative to $T$, $w,x,y$ and $z$ lie in a connected component of $X^{\Ker(\alpha)^0}$ isomorphic to $\F_0$ and $w\geq x, y \geq z$ are ordered by their corresponding weights. 
		\item $\rho_{n/2}c_1^T(L_{\alpha})(f_y-f_z)+\rho_{-n/2}c_1^T(L_{\alpha})(f_w-f_x)\equiv 0\mod c_1^T(L_{\alpha})^2$ whenever $\alpha$ is a positive root of $G$ relative to $T$, $w,x,y$ and $z$ lie in a connected component of $X^{\Ker(\alpha)^0}$ isomorphic to a rational ruled surface $\F_n$, $n\geq 1$, and $w\geq x\geq y \geq z$ are ordered by their corresponding weights. 
	\end{enumerate}
\end{sat}
\begin{bem}\label{Condition (i)_Theorem}
	We will see later in the proof that condition (i) in the preceding proposition comes from Proposition \ref{Krishna7.8}. Further, the formulation of the equations in the conditions (ii) and (iv) slightly differs from the one in Brion's description. We need to introduce the terms $\rho_{n/2}$, $n\in \Z\setminus\{0\}$, from Definition \ref{Reminder_Quotient} because of the universal formal group law in cobordism. Besides that, as opposed to the formulation of Brion, we have to distinguish between the cases $\F_0$ and $\F_n$, $n\geq 1$, again due to the universal formal group law. In the case of a smooth projective spherical $G$-variety $X$, the theorem is a generalisation of Proposition \ref{Krishna7.8} because the cases (ii)-(iv) do not occur if the variety has only finitely many $T$-stable curves. 
\end{bem}
Now we want to compute equivariant cobordism for projective planes and Hirzebruch surfaces. Therefore, we describe the irreducible components of $X^{T'}$ for singular codimension one subtori $T'$ coming from Proposition \ref{componentsX^T'} in some more detail.

We start with the description of the $T$-fixed points in $X^{T'}$. Let $D$ be the torus of diagonal matrices in $\SL_2$ and let $\alpha$ be the positive root. At first, we want to consider the two cases of $\PZ(V)$ for a non-trivial $\SL_2$-module $V$ of dimension three. Set $V_{n+1}:=\Sym^{n+1}(k^2)$. Let $V=V_0\oplus V_1$ be the first non-trivial $\SL_2$-module of dimension three. The weights of $D$ in $V$ are $-\alpha/2,0$ and $\alpha/2$ with the induced group action of $D$ on $V$. We denote by $x,y$ and $z$ the corresponding fixed points of $D$ in $\PZ(V)$. To be more explicit, the fixed points $x=[1:0:0],y=[0:1:0]$ and $z=[0:0:1]$ correspond to the weights $\alpha/2,0,-\alpha/2$, respectively. Thus, we identify $\Omega^\ast_D(\PZ(V)^D)_\Q$ with $S(D)_\Q^3$. 

Similarly, for the second non-trivial $\SL_2$-module $V=V_2=\mathfrak{sl_2}$ of dimension three, the corresponding weights are $\alpha,0$ and $-\alpha$ where the corresponding fixed points are again $x=[1:0:0], y=[0:1:0]$ and $z=[0:0:1]$, respectively. 

Next, we consider the case $\F_0=\PZ^1\times \PZ^1$ with $D$-action given by 
\begin{align*}
	d\cdot ([a:b],[u:v])=([da:d^{-1}b],[du:d^{-1}v]).
\end{align*} We denote by $w$ and $z$ the $D$-fixed points $([1:0],[1:0])$ and $([0:1],[0:1])$, respectively. Further, we denote the remaining two $D$-fixed points $([1:0],[0:1])$ and $([0:1],[1:0])$ by $x$ and $y$, respectively. 

Lastly, we have a look at the rational ruled surfaces $\F_n$, $n\geq 1$, which is the closure of the $\SL_2$-orbit $\SL_2\cdot [v_1+v_{n+1}]$ in $\PZ(V)$ for $V:=V_1\oplus V_{n+1}$ where $v_1\in V_1$ and $v_{n+1}\in V_{n+1}$ denote the two highest weight vectors, respectively.  We recall that $\F_n$ has four $D$-fixed points $w,x,y$ and $z$ with corresponding weights $(n+1)\alpha/2, \alpha/2, -\alpha/2$ and $-(n+1)\alpha/2$, respectively, by the induced $D$-action on $\F_n$. Therefore, we can identify $\Omega^\ast_D(\F_n^D)_\Q$ with $S(D)_\Q^4$.    
\begin{thm}\label{Brion 7.2.Cob}
	Let $X$ be a Hirzebruch surface $\F_n$ or a projective plane $\PZ(V)$ as above. 
	\begin{enumerate}[(i)]
		\item The image of the pullback
		\begin{align*}
			i^\ast: \Omega^\ast_D(\F_n)_\Q\to S(D)^4_\Q
		\end{align*}
		consists of all $(f_w,f_x,f_y,f_z)\in S(D)^4_\Q$ such that 
		\begin{align*}
			f_w\equiv f_x\equiv f_y\equiv f_z&\mod c_1^D(L_{\alpha})\text{ and }\\
			f_w-f_x-f_y+f_z\equiv 0&\mod c_1^D(L_{\alpha})^2
		\end{align*}
		hold for $n=0$ and of all $(f_w,f_x,f_y,f_z)\in S(D)^4_\Q$ such that
		\begin{align*}
			f_w\equiv f_x\equiv f_y\equiv f_z&\mod c_1^D(L_{\alpha})\text{ and }\\
			\rho_{n/2}c_1^D(L_{\alpha})(f_y-f_z)+\rho_{-n/2}c_1^D(L_{\alpha})(f_w-f_x)\equiv 0&\mod c_1^D(L_{\alpha})^2
		\end{align*}
		hold for $n\geq 1$.
		\item Moreover, the image of 
		\begin{align*}
			\Omega^\ast_D(\PZ(V))_\Q\to S(D)^3_\Q
		\end{align*}
		consists of all $(f_x,f_y,f_z)$ such that 
		\begin{align*}
			f_x\equiv f_y\equiv f_z&\mod c_1^D(L_{\alpha})\text{ and }\\
			(f_x-f_y)+\rho_{1/2}c_1^D(L_{\alpha})(f_z-f_x)\equiv 0&\mod c_1^D(L_{\alpha})^2
		\end{align*}
		hold.
	\end{enumerate} 
\end{thm} 
\begin{bem}
	In the above statement, the is one equation more than needed to keep the symmetry in the arguments. For example, in the $\F_n$ case we could remove the equation $f_w\equiv f_x \mod c_1^D(L_\alpha)$. 
\end{bem}
\begin{proof}
	We first consider the case of $\PZ(V)$ for $V=V_0\oplus V_1$. Since $i^\ast$ is a ring homomorphism, the class $[\PZ(V)\to \PZ(V)]$ maps to $(1,1,1)$. Now we want to compute the images of the closures of the Bialynicki-Birula cells, i.e. the images of the equivariant fundamental classes of the $D$-stable cobordism cycles $[(yz)\to \PZ(V)]$ and $[z\to \PZ(V)]$. We have a look at the pullback 
	\begin{align*}
	i^\ast[Y\to \PZ(V)]=(i^\ast_x[Y\to \PZ(V)],i^\ast_y[Y\to \PZ(V)],i^\ast_z[Y\to \PZ(V)])
	\end{align*}where $i^\ast_x[Y\to \PZ(V)]$ denotes the pullback of the class $[Y\to \PZ(V)]$ under the inclusion $i_x$ of the corresponding fixed point in $\PZ(V)$. To compute $i_z^\ast[Y\to \PZ(V)]$ we can replace $\PZ(V)$ by any open $D$-stable neighbourhood $U_z$ of $z$. In this case we choose $U_z$ to be the affine chart of $\PZ(V)$ in which the coordinate associated to $z$ does not vanish. We introduce the coordinates $a,b$ and $c$ for $V$ such that our coordinates for $U_z$ become $a/c$ and $b/c$. Therefore, $D$ acts linearly on $U_z$ with weights $\alpha$ and $\alpha/2$. We choose $f(a/c,b/c)=a/c$ which is an eigenfunction of $D$ with respect to weight $\alpha$. In this situation, we can apply Corollary \ref{Relation_We_actually_use} because $f$ defines a global section $s$ which is transverse to the zero section with zero-subscheme $Z_0=(yz)\cap U_z$. Thus, we know that
	\begin{align*}
	[(yz)\cap U_z\to U_z]=c_1^D(L_{\alpha})[\Spec k\to \Spec k]\cdot [U_z\to U_z]
	\end{align*}holds in $\Omega^\ast_D(U_z)_\Q$. Pulling back to $\Omega^\ast_D(z)_\Q$ yields $i^\ast_z[(yz)\cap U_z\to U_z]=c_1^D(L_{\alpha})$. We can apply the same argument for the pullback $i^\ast_y[(yz)\to \PZ(V)]$ by choosing $U_y$ to be the open affine neighbourhood of $y$ such that the coordinate associated to $y$ does not vanish. Thus, $D$ acts linearly on $U_y$ with weights $\alpha/2$ and $-\alpha/2$. We take $f(a/b,c/b)=a/b$ which is an eigenfunction of $D$ with respect to weight $\alpha/2$. Therefore, we conclude $i^\ast_y[(yz)\cap U_y]=c_1^D(L_{\alpha/2})$ by the same argument as above.
	
	Finally, we consider the last closure of the Bialynicki-Birula cells, i.e. the point $z$.  
	Clearly, $z$ is the complete intersection of the two lines $(yz)$ and $(xz)$. 
	Therefore, we want to compute the pullback of $[(xz)\cap (yz)\cap U_z\to U_z]=[z\to U_z]$. We want to apply the same argument again using the relation 
	\begin{align*}
	c_1^D(L_{\alpha/2})\cdot [(yz)\cap U_z\to U_z]=[(xz)\cap (yz)\cap U_z\to U_z]=[z\to U_z]
	\end{align*} 
	from Corollary \ref{Relation_We_actually_use} where $z=(xz)\cap (yz)\cap U_z$ is the zero-subscheme of the section defined by the eigenfunction $g(a/c,b/c)=b/c$ of $D$ with respect to weight $\alpha/2$ on $(yz)\cap U_z$. Using the equality $i^\ast_z[(yz)\cap U_z\to U_z]=c_1^D(L_{\alpha})$, we obtain the pullback
	\begin{align*}
	i^\ast_z[z\to U_z]=c_1^D(L_{\alpha/2})\cdot i^\ast_z[(yz)\cap U_z\to U_z]=c_1^D(L_{\alpha/2})\cdot c_1^D(L_{\alpha})
	\end{align*}  
	in $\Omega^\ast_D(z)_\Q$. The images of the $D$-stable cobordism cycles coming from the closures of the Bialynicki-Birula decomposition generate the equivariant cobordism ring by \cite[Corollary 4.8]{CobTorus}. Therefore, the image of the pullback $i^\ast: \Omega^\ast_D(\PZ(V))_\Q\to S(D)^3_\Q$ is generated by the images 
	\begin{align*}
	[\PZ(V)\to \PZ(V)]&\mapsto (1,1,1)\\
	[(yz)\to \PZ(V)]&\mapsto (0,c_1^D(L_{\alpha/2}),c_1^D(L_{\alpha}))\\
	[z\to \PZ(V)]&\mapsto (0,0,c_1^D(L_{\alpha/2})c_1^D(L_{\alpha})).
	\end{align*}
	These images satisfy the given equations which can be seen by again expressing $c_1^D(L_{\alpha/2})$ as a formal power series in the variable $c_1^D(L_\alpha)$ with rational coefficients. For the following computation and similar ones upcoming in the sequel of this proof, we remark that any element which is divisible by $c_1^D(L_{\alpha})$ will be also divisible by $c_1^D(L_{n\alpha/m})$ for $m,n\in \Z\setminus \{0\}$ because we can again express the first Chern class in terms of the second one and factor out. Therefore, for an element $(f_x,f_y,f_z)\in S(T)_\Q^3$ satisfying the given equations we have
	\begin{align*}
	(f_x,f_y,f_z)&=f_x(1,1,1)+(0,f_y-f_x,f_z-f_x)\\
	&=f_x(1,1,1)+\frac{f_y-f_x}{c_1^D(L_{\alpha/2})}(0,c_1^D(L_{\alpha/2}),c_1^D(L_\alpha))\\
	&+(0,0,(f_x-f_y)\frac{c_1^D(L_\alpha)}{c_1^D(L_{\alpha/2})}+f_z-f_x)\\
	&=f_x(1,1,1)+\frac{f_y-f_x}{c_1^D(L_{\alpha/2})}(0,c_1^D(L_{\alpha/2}),c_1^D(L_\alpha))\\
	&+\left(0,0,c_1^D(L_{\alpha/2})c_1^D(L_{\alpha})\left(\frac{(f_x-f_y)c_1^D(L_\alpha)+c_1^D(L_{\alpha/2})(f_z-f_x)}{c_1^D(L_{\alpha/2})^2c_1^D(L_\alpha)}\right)\right)\\
	&=f_x(1,1,1)+\frac{f_y-f_x}{c_1^D(L_{\alpha/2})}(0,c_1^D(L_{\alpha/2}),c_1^D(L_\alpha))\\
	&+\frac{(f_x-f_y)c_1^D(L_\alpha)+c_1^D(L_{\alpha/2})(f_z-f_x)}{c_1^D(L_{\alpha/2})^2c_1^D(L_\alpha)}(0,0,c_1^D(L_{\alpha/2})c_1^D(L_{\alpha}))
	\end{align*} 
	which completes the proof in the case $V=V_0\oplus V_1$.
	
	The computation for $V=V_2$ can be done similarly. We obtain
	\begin{align*}
	i^\ast: \Omega^\ast_D(\PZ(V))_\Q&\to S(D)^3_\Q\\
	[\PZ(V)\to \PZ(V)]&\mapsto (1,1,1)\\
	[(yz)\to \PZ(V)]&\mapsto (0,c_1^D(L_\alpha),c_1^D(L_{2\alpha}))\\
	[z\to \PZ(V)]&\mapsto (0,0,c_1^D(L_\alpha)c_1^D(L_{2\alpha}))
	\end{align*}
	which satisfy the given equations using the properties of the formal group law. Again, we obtain 
	\begin{align*}
	(f_x,f_y,f_z)&=f_x(1,1,1)+(0,f_y-f_x,f_z-f_x)\\
	&=f_x(1,1,1)+\frac{f_y-f_x}{c_1^D(L_{\alpha})}(0,c_1^D(L_{\alpha}),c_1^D(L_{2\alpha}))\\
	&+(0,0,(f_x-f_y)\frac{c_1^D(L_{2\alpha})}{c_1^D(L_{\alpha})}+f_z-f_x)\\
	&=f_x(1,1,1)+\frac{f_y-f_x}{c_1^D(L_{\alpha})}(0,c_1^D(L_{\alpha}),c_1^D(L_{2\alpha}))\\
	&+\left(0,0,c_1^D(L_{\alpha})c_1^D(L_{2\alpha})\left(\frac{(f_x-f_y)c_1^D(L_{2\alpha})+c_1^D(L_{\alpha})(f_z-f_x)}{c_1^D(L_{\alpha})^2c_1^D(L_{2\alpha})}\right)\right)\\
	&=f_x(1,1,1)+\frac{f_y-f_x}{c_1^D(L_{\alpha})}(0,c_1^D(L_{\alpha}),c_1^D(L_{2\alpha}))\\
	&+\frac{(f_x-f_y)c_1^D(L_{2\alpha})+c_1^D(L_{\alpha})(f_z-f_x)}{c_1^D(L_{\alpha})^2c_1^D(L_{2\alpha})}(0,0,c_1^D(L_{\alpha})c_1^D(L_{2\alpha}))
	\end{align*}
	which completes the proof for $V=V_2$ since the last coefficient is well-defined using the properties of the formal group law and the given equations. More precisely, the quotient $c_1^D(L_\alpha)/c_1^D(L_{2\alpha})$ has the same coefficients as $\rho_{1/2}c_1^D(L_\alpha)$ and the only difference will be the variable $c_1^D(L_{2\alpha})$ in the first quotient as opposed to $c_1^D(L_\alpha)$ in the second one. As we consider the reduction modulo $c_1^D(L_\alpha)^2$, we only need to take the first two summands of $c_1^D(L_\alpha)/c_1^D(L_{2\alpha})$ into account. Therefore, $c_1^D(L_\alpha)/c_1^D(L_{2\alpha})$ differs from $\rho_{1/2}c_1^D(L_\alpha)$ only by a factor of two in the second summand. Their difference contains a factor $c_1^D(L_\alpha)$ which will be multiplied by $(f_z-f_x)$. This product vanishes modulo $c_1^D(L_\alpha)^2$ because of the first equation and thus, we reduced the coefficient to the known equation $(f_x-f_y)+\rho_{1/2}c_1^D(L_{\alpha})(f_z-f_x)$ which finishes the argument.  
	
	Next, we consider the case $\F_0=\PZ^1\times \PZ^1$ for which we choose $U_w$ to be an open $D$-stable neighbourhood of $w=([1:0];[1:0])$. We get $(t^{-2}b/a,t^{-2}v/u)$ for coordinates $([a:b];[u:v])$ which implies that $D$ acts linearly on $U_w$ with weight $-\alpha$. The class $[\F_0]\in \Omega^\ast_D(\F_0)_\Q$ again maps to $(1,1,1,1)$ and we want to compute the remaining images of the closures of the Bialynicki-Birula cells.
	
	Therefore, we take the closure $(wx)$ of one of the remaining Bialynicki-Birula cells. We choose $f(b/a,v/u)=b/a$ to be an eigenfunction of $D$ with respect to weight $-\alpha$. By Corollary \ref{Relation_We_actually_use} we obtain
	\begin{align*}
	[(wx)\cap U_w\to U_w]=c_1^D(L_{-\alpha})[\Spec k\to \Spec k]\cdot [U_w\to U_w]
	\end{align*}
	in $\Omega^\ast_D(U_w)_\Q$. Pulling this relation back yields $i^\ast_w[(wx)\cap U_w\to U_w]=c_1^D(L_{-\alpha})$. With the eigenfunction $f(b/a,u/v)=b/a$ and an open $D$-stable neighbourhood $U_x$ of the fixed point $x$ we obtain $i^\ast_x[(wx)\cap U_x\to U_x]=c_1^D(L_{-\alpha})$.
	
	For the pullbacks of $(wy)$ we take the eigenfunction $f(b/a,v/u)=v/u$ of $D$ with respect to weight $-\alpha$ on the open $D$-stable $U_w$ from above, but in this case we have $V(f)=(wy)\cap U_w$ and therefore, $i^\ast_w[(wy)\cap U_w\to U_w]=c_1^D(L_{-\alpha})$.
	Similarly, we obtain $i^\ast_y[(wy)\cap U_y\to U_y]=c_1^D(L_{-\alpha})$.
	
	Lastly, we consider the pullback of the point $w$ which is again the complete intersection of $(wy)$ and $(wx)$. By the same argument as in the above cases, we get 
	\begin{align*}
	i^\ast_w[w\to U_w]&=i^\ast_w[(wx)\cap (wy)\cap U_w\to U_w]\\&=c_1^D(L_{-\alpha})\cdot i^\ast_w[(wy)\cap U_w\to U_w]\\&=c_1^D(L_{-\alpha})c_1^D(L_{-\alpha})
	\end{align*}
	whereas the other pullbacks of the class of the point $w$ vanish. We summarise that the image is given by
	\begin{align*}
	i^\ast: \Omega^\ast_D(\F_0)_\Q&\to S(D)^4_\Q\\
	[\F_0\to \F_0]&\mapsto (1,1,1,1)\\
	[(wx)\to \F_0]&\mapsto (c_1^D(L_{-\alpha}),c_1^D(L_{-\alpha}),0,0)\\
	[(wy)\to \F_0]&\mapsto (c_1^D(L_{-\alpha}),0,c_1^D(L_{-\alpha}),0)\\
	[w\to \F_0]&\mapsto (c_1^D(L_{-\alpha})c_1^D(L_{-\alpha}),0,0,0).
	\end{align*}
	which satisfies the equations. 
	
	Conversely, for an element $(f_w,f_x,f_y,f_z)\in S(T)^4_\Q$ fulfilling the conditions we have
	\begin{align*}
	(f_w,f_x,f_y,f_z)&=f_z(1,1,1,1)+(f_w-f_z,f_x-f_z,f_y-f_z,0)\\
	&=f_z(1,1,1,1)+\frac{f_y-f_z}{c_1^D(L_{-\alpha})}(c_1^D(L_{-\alpha}),0,c_1^D(L_{-\alpha}),0)\\
	&+(f_w-f_y,f_x-f_z,0,0)\\
	&=f_z(1,1,1,1)+\frac{f_y-f_z}{c_1^D(L_{-\alpha})}(c_1^D(L_{-\alpha}),0,c_1^D(L_{-\alpha}),0)\\
	&+\frac{f_x-f_z}{c_1^D(L_{-\alpha})}(c_1^D(L_{-\alpha}),c_1^D(L_{-\alpha}),0,0)+(f_w-f_x-f_y+f_z,0,0,0)\\
	&=f_z(1,1,1,1)+\frac{f_y-f_z}{c_1^D(L_{-\alpha})}(c_1^D(L_{-\alpha}),0,c_1^D(L_{-\alpha}),0)\\
	&+\frac{f_x-f_z}{c_1^D(L_{-\alpha})}(c_1^D(L_{-\alpha}),c_1^D(L_{-\alpha}),0,0)\\
	&+\frac{f_w-f_x-f_y+f_z}{c_1^D(L_{-\alpha})^2}(c_1^D(L_{-\alpha})^2,0,0,0)
	\end{align*}
	which completes the proof for the case $\F_0$.
	
	In the following, we consider the case $\F_n$ for $n\geq 1$. The class $[\F_n]\in \Omega^\ast_D(\F_n)_\Q$ is again mapped to $(1,1,1,1)$. 
	
	Now we compute the remaining pullbacks of the closures of the Bialynicki-Birula cells. We choose again an open $D$-stable neighbourhood $U_w$ of the fixed point $w=[0:0:1:0:...:0]$. The induced $D$-action on $U_w$ is given by $(t^{-n}x_0/y_0,t^{-2}y_1/y_0)$ for coordinates $[x_0:x_1:y_0:y_1:...:y_{n+1}]$ and therefore, $D$ acts linearly on $U_w$ with weights $-n\alpha/2$ and $-\alpha$. We choose $f(x_0/y_0,y_1/y_0)=y_1/y_0$ to be an eigenfunction of $D$ with respect to weight $-\alpha$. By the relations on the coordinates in $\F_n$ we obtain $V(f)=(wx)\cap U_w$ with the given notations of the $D$-fixed points. As above, we get $i^\ast_w[(wx)\cap U_w\to U_w]=c_1^D(L_{-\alpha})$. One may observe that the pullback does not depend on the choice of coordinates for $U_w$. For the point $x=[1:0:...:0]$ and a $D$-stable nighbourhood $U_x$ we choose the eigenfunction $f(x_1/x_0,y_1/x_0)=x_1/x_0$ of $D$ with respect to weight $-\alpha$ which leads to $i^\ast_x[(wx)\cap U_x\to U_x]=c_1^D(L_{-\alpha})$. 
	
	For the pullback of $(xy)$ let $U_x$ be given by coordinates $(y_0/x_0,x_1/x_0)$ and take the eigenfunction $f(y_0/x_0,x_1/x_0)=y_0/x_0$ of weight $n\alpha/2$ which leads to $V(f)=(xy)\cap U_x$ and therefore to $i^\ast_x[(xy)\cap U_x\to U_x]=c_1^D(L_{n\alpha/2})$. For the coordinates $(y_0/x_1,y_{n+1}/x_1)$ for $U_y$ and the eigenfunction $f(y_0/x_1,y_{n+1}/x_1)=y_{n+1}/x_1$ we get $V(f)=(xy)\cap U_y$ and hence $i^\ast_y[(xy)\cap U_y\to U_y]=c_1^D(L_{-n\alpha/2})$.
	
	Finally, we consider the pullback of the point $w$ by introducing an eigenfunction on $(wx)\cap U_w$. We choose $g(x_0/y_0,y_1/y_0)=x_0/y_0$ which is an eigenfunction of weight $-n\alpha/2$. This leads to $V(g)=(wz)\cap (wx)\cap U_w=w$ and thus we obtain 
	\begin{align*}
	[w\to U_w]=c_1^D(L_{-n\alpha/2})[(wx)\cap U_w\to U_w]
	\end{align*}
	in $\Omega^\ast_D(U_w)_\Q$ again by Corollary \ref{Relation_We_actually_use}. 
	
	We conclude $i_w^\ast[w\to U_w]=c_1^D(L_{-n\alpha/2})c_1^D(L_{-\alpha})$ and obtain the image
	\begin{align*}
	i^\ast: \Omega^\ast_D(\F_n)_\Q&\to S(D)^4_\Q\\
	[\F_n\to \F_n]&\mapsto (1,1,1,1)\\
	[(wx)\to \F_n]&\mapsto (c_1^D(L_{-\alpha}),c_1^D(L_{-\alpha}),0,0)\\
	[(xy)\to \F_n]&\mapsto (0,c_1^D(L_{n\alpha/2}),c_1^D(L_{-n\alpha/2}),0)\\
	[w\to \F_n]&\mapsto (c_1^D(L_{-\alpha})c_1^D(L_{-n\alpha/2}),0,0,0)
	\end{align*} 
	which satisfies the given equations. 
	
	Conversely, let $(f_w,f_x,f_y,f_z)\in S(T)^4_\Q$ be an element fulfilling the conditions. This leads to 
	\begin{align*}
	&(f_w,f_x,f_y,f_z)=f_z(1,1,1,1)+(f_w-f_z,f_x-f_z,f_y-f_z,0)\\
	&=f_z(1,1,1,1)+\frac{f_y-f_z}{c_1^D(L_{-n\alpha/2})}(0,c_1^D(L_{n\alpha/2}),c_1^D(L_{-n\alpha/2}),0)\\
	&+(f_w-f_z,\frac{(f_z-f_y)c_1^D(L_{n\alpha/2})}{c_1^D(L_{-n\alpha/2})}+f_x-f_z,0,0)\\
	&=f_z(1,1,1,1)+\frac{f_y-f_z}{c_1^D(L_{-n\alpha/2})}(0,c_1^D(L_{n\alpha/2}),c_1^D(L_{-n\alpha/2}),0)\\
	&+\left(\frac{(f_z-f_y)c_1^D(L_{n\alpha/2})+(f_x-f_z)c_1^D(L_{-n\alpha/2})}{c_1^D(L_{-n\alpha/2})c_1^D(L_{-\alpha})}\right)(c_1^D(L_{-\alpha}),c_1^D(L_{-\alpha}),0,0)\\
	&+(\frac{(f_x-f_w)c_1^D(L_{n\alpha/2})+(f_y-f_z)c_1^D(L_{-n\alpha/2})}{c_1^D(L_{-n\alpha/2})},0,0,0)\\
	&=f_z(1,1,1,1)+\frac{f_y-f_z}{c_1^D(L_{-n\alpha/2})}(0,c_1^D(L_{n\alpha/2}),c_1^D(L_{-n\alpha/2}),0)\\
	&+\left(\frac{(f_z-f_y)c_1^D(L_{n\alpha/2})+(f_x-f_z)c_1^D(L_{-n\alpha/2})}{c_1^D(L_{-n\alpha/2})c_1^D(L_{-\alpha})}\right)(c_1^D(L_{-\alpha}),c_1^D(L_{-\alpha}),0,0)\\
	&+\left(\frac{(f_y-f_z)c_1^D(L_{n\alpha/2})+c_1^D(L_{-n\alpha/2})(f_w-f_x)}{c_1^D(L_{-n\alpha/2})^2c_1^D(L_{-\alpha})}\right)(c_1^D(L_{-\alpha})c_1^D(L_{-n\alpha/2}),0,0,0)
	\end{align*}
	which completes the proof in the case $\F_n$ because of the equations and the above mentioned fact that an element which is divisible by $c_1^D(L_{\alpha})$ will be also divisible by $c_1^D(L_{n\alpha/m})$ for $m,n\in \Z\setminus\{0\}$ since we consider rational coefficients. 
\end{proof}
\begin{bem}
	The equations given in Proposition \ref{Brion 7.2.Cob} reduce to Brion's equations given in \cite[Proposition 7.2]{BrionTorusActions} for rational equivariant Chow rings. In order to be able to compute rational equivariant cobordism rings one has to consider the universal formal group law and not the additive formal group law which simplifies the computations in the Chow group case.  
\end{bem}
Next, we want to prove Theorem \ref{EquivCobThmSpherical} which is a refinement of \cite[Theorem 7.3]{BrionTorusActions}. One way to prove it would be to give explicit generators and relations describing the equivariant algebraic cobordism module as in the Chow group case (cf. \cite[Theorem 2.1]{BrionTorusActions}) which is not known at present. Luckily, we do not need such a deep result in order to be able to prove Theorem \ref{EquivCobThmSpherical}. In our situation it will be enough to use some known results on $T$-filtrable varieties and their equivariant algebraic cobordism rings. 
\begin{proof}[Proof of Theorem \ref{EquivCobThmSpherical}]\label{Proof_of_Theorem_Cobordism_Spherical}
	We want to apply Proposition \ref{EquivCobIntersectionImages} in order to compute $\Omega_T^\ast(X)_\Q$. Due to Proposition \ref{componentsX^T'} we know which fixed point subschemes $X^{T'}$ can occur and therefore, we distinguish between codimension one subtori $T'$ with $\dim X^{T'}\leq 1$ and those with $\dim X^{T'}=2$. 
	
	For a subtorus $T'$ with $\dim X^{T'}\leq 1$ there are only finitely many $T$-stable curves in $X^{T'}$ and furthermore, in the setting of a smooth projective spherical $G$-variety $X$, we have only finitely many $T$-fixed points by \cite[Lemma 2.2]{ConciniSpringer}. This implies that the assumptions of Proposition \ref{Krishna7.8} are fulfilled and thus, we can apply Proposition \ref{Krishna7.8} to $X^{T'}$ which leads to case (i).
	
	Now we consider the case where $\dim X^{T'}=2$ for which we know that $X^{T'}$ is either a projective plane or a Hirzebruch surface $\F_n$. The $T$-orbits in $X^{T'}$ are always one-dimensional and thus, the surfaces occurring in (ii)-(iv) must consist of infinitely many $T$-stable curves. For these cases we need some different results. We claim that $\Omega_\ast^T(X^{T'})_\Q\cong \Omega_\ast
	^{T/F(\alpha)}(X^{T'})_\Q$ holds where $F(\alpha)$ is given as in Remark \ref{PositiveRootsSubtorus}. We will use \cite[Theorem 4.7]{CobTorus} in order to prove our claim. This theorem states that we have an isomorphism of $S(T)$-modules $\Omega_\ast^T(X^{T'})\cong \Omega_\ast(X^{T'})[[t_1,...,t_r]]_{\gr}$ since $T$ is acting on the $T$-filtrable variety $X^{T'}$ where $r$ is the rank of $T$ and $t_i$ corresponds to $c_1^T(L_{\chi_i})$ for a chosen basis of the character group of $T$. We remark that $X^{T'}$ is also a $T/F(\alpha)$-filtrable variety as $F(\alpha)$ acts trivially on $X^{T'}$ and therefore the $T$-action factors through the $T/F(\alpha)$-action. Since $T/F(\alpha)$ is similarly a torus of rank $r$ acting with the same action on $X^{T'}$ we obtain the isomorphism $\Omega_\ast^{T/F(\alpha)}(X^{T'})\cong \Omega_\ast(X^{T'})[[t_1,...,t_r]]_{\gr}$ of $S(T/F(\alpha))$-modules where $t_i$ here corresponds to $c_1^{T/F(\alpha)}(L_{\chi_i'})$ for the corresponding basis of the character group of $T/F(\alpha)$, but as we are considering rational coefficients we have $S(T)_\Q\cong S(T/F(\alpha))_\Q$ by Lemma \ref{IsomorphismCobordism T and T/F}. This implies the claim and using the same argument for the torus $T'\times S_m(\alpha)$ we obtain
	\begin{align*}
	\Omega_\ast^T(X^{T'})_\Q&\cong \Omega_\ast^{(T'\times S_m(\alpha))/(F(\alpha)\times F(\alpha))}(X^{T'})_\Q\\
	&\cong \bigoplus_{i\in\Z}\Omega_i^{T'\times S_m(\alpha)}(X^{T'})_\Q\\
	&\cong \bigoplus_{i\in\Z}\varprojlim_{j}\Omega_i((\Spec k\times U^2_j\times X^{T'}\times U^1_j)/(T'\times S_m(\alpha)))_\Q\\
	&\cong \bigoplus_{i\in\Z}\varprojlim_{j}\Omega_i((\Spec k\times U^2_j)/T'\times (X^{T'}\times U^1_j)/S_m(\alpha))_\Q\\
	&\cong \bigoplus_{i\in\Z}\varprojlim_{j}\bigoplus_{i_1+i_2=i}\Omega_{i_1}((\Spec k\times U_j^2)/T')_\Q\otimes_{\La_\Q}\Omega_{i_2}((X^{T'}\times U^1_j)/S_m(\alpha))_\Q\\
	&\cong \bigoplus_{i\in\Z}\bigoplus_{i_1+i_2=i}\varprojlim_{j}\Omega_{i_1}((\Spec k\times U_j^2)/T')_\Q\otimes_{\La_\Q}\Omega_{i_2}((X^{T'}\times U^1_j)/S_m(\alpha))_\Q\\
	&\cong \bigoplus_{i\in\Z}\bigoplus_{i_1+i_2=i}\Omega_{i_1}^{T'}(\Spec k)_\Q\otimes_{\La_\Q}\Omega_{i_2}^{S_m(\alpha)}(X^{T'})_\Q\\
	&\cong \Omega_\ast^{T'}(\Spec k)_\Q\otimes_{\La_\Q}\Omega_\ast^{S_m(\alpha)}(X^{T'})_\Q
	\end{align*}
	where $U^1_j$ and $U^2_j$ are the corresponding parts of the sequences of good pairs $\{(V^1_j,U^1_j)\}_{j\geq 0}$ and $\{(V^2_j,U^2_j)\}_{j\geq 0}$ for $S_m(\alpha)$ and $T'$, respectively. In this case we know that $U_j^2/T'$ are products of projective spaces by the choice of good pairs in the proof of \cite[Lemma 6.1]{CobTorus}. As a product of projective spaces, the $U_j^2/T'$ are cellular which means that we can use a special version of a Künneth formula (cf. \cite[Proposition 7]{HellerMalagon-Lopez}) from line 4 to 5. The ordinary cobordism does not exist for negative degrees and therefore the inverse limit and the finite sum commute in our setting. We conclude the result by Proposition \ref{Brion 7.2.Cob}. 
\end{proof}
\begin{thm}\cite[Theorem 3.4]{CobTorus}\label{CobTorus3.4}
	Let $T$ be a torus acting on a $k$-variety $X$. Then there is an isomorphism
	\begin{align*}
	\overline{r}^T_X: \Omega^T_\ast(X)\otimes_{S(T)}\La\overset{\cong}{\longrightarrow}\Omega_\ast(X).
	\end{align*}
	If $X$ is smooth, this is an $\La$-algebra isomorphism.
\end{thm}
\begin{bem}
	The previous result leads to an abstract description of the rational ordinary algebraic cobordism ring of any smooth projective and spherical $G$-variety $X$. Using this result, we would be able to describe $\Omega^\ast(X)_\Q$ explicitly if we could compute all the classes in $\Omega_T^\ast(X)_\Q$. We will come back to this problem in Section \ref{Section_equiv_Multiplicities}.
\end{bem}
Now we want to have a look at the specific example $\IG(2,5)$ for which we can use Theorem \ref{EquivCobThmSpherical} in order to compute its rational equivariant cobordism ring.
\begin{bei}\label{ExampleIG(2,5)}
	Let $V=k^5$ be given with standard basis $e_1,...,e_5$. Recall that the odd symplectic Grassmannian $X=\IG(2,5)$ is given by 
	\begin{align*}
		\IG(2,5)=\{\Sigma\in \Gr(2,5)\ \vert \ \Sigma \text{ is isotropic for } \omega\}
	\end{align*}
	where $\omega$ is the antisymmetric form given by 
	\begin{align*}
		\omega: V\times V\to k, ((a_i),(b_j))_{1\leq i,j\leq 5}\mapsto a_5b_1+a_4b_2-a_2b_4-a_1b_5
	\end{align*}
	which has kernel $e_3$.	It is well-known that all odd symplectic Grassmannians are smooth, projective and horospherical (cf. \cite[Theorem 0.1]{Pasquier_horospherical_classification}). We will consider the natural torus action of $T\subseteq\Sp_4$ on $\IG(2,5)$ which leads to eight $T$-fixed points in $\IG(2,5)$. These are given by
	\begin{align*}
	x_{12}&=[e_1\land e_2], \ x_{13}=[e_1\land e_3], \  x_{14}=[e_1\land e_4],\ x_{23}=[e_2\land e_3]\\ 
	x_{25}&=[e_2\land e_5], \ x_{34}=[e_3\land e_4],\ x_{35}=[e_3\land e_5],\ x_{45}=[e_4\land e_5]
	\end{align*}  
	in $\PZ(\bigwedge^2 k^5)$ whereas the points $x_{15}=[e_1\land e_5]$ and $x_{24}=[e_2\land e_4]$ are not in $\IG(2,5)$. 
	
	The positive roots of $(\Sp_4,T)$ are given by $\varepsilon_1-\varepsilon_2, \varepsilon_1+\varepsilon_2,2\varepsilon_1$ and $2\varepsilon_2$. A short computation shows that the fixed point subschemes $X^{\Ker(\varepsilon_1-\varepsilon_2)^0}$ and $X^{\Ker(\varepsilon_1+\varepsilon_2)^0}$ consist of three $T$-stable curves and the remaining two isolated fixed points, respectively. Lastly, the fixed point subschemes $X^{\Ker(2\varepsilon_1)^0}$ and $X^{\Ker(2\varepsilon_2)^0}$ consist of two projective planes $\PZ^2$ and a $T$-stable curve, respectively. 
	
	The other codimension one subtori given by $T'=\Ker(\chi)^0$ for some primitive character $\chi$ of $T$ which is not a root will not contribute to the computations of cobordism since $X^{T'}=X^T$ holds for those $T'$. 
	
	These precise descriptions of the fixed point subschemes lead to the equations describing the image of $i^\ast: \Omega_T^\ast(\IG(2,5))_\Q\to \Omega_T^\ast((\IG(2,5))^T)_\Q$. Using Theorem \ref{EquivCobThmSpherical} the equations are given by 
	\begin{align*}
	&f_{13}\equiv f_{23}\mod c_1^T(L_{\varepsilon_1-\varepsilon_2}), \ \ f_{34}\equiv f_{35}\mod c_1^T(L_{\varepsilon_1-\varepsilon_2}), \\
	&f_{14}\equiv f_{25}\mod c_1^T(L_{\varepsilon_1-\varepsilon_2}), \ \ f_{13}\equiv f_{34}\mod c_1^T(L_{\varepsilon_1+\varepsilon_2}), \\
	&f_{23}\equiv f_{35}\mod c_1^T(L_{\varepsilon_1+\varepsilon_2}), \ \ f_{12}\equiv f_{45}\mod c_1^T(L_{\varepsilon_1+\varepsilon_2}), \\
	&f_{13}\equiv f_{35}\mod c_1^T(L_{2\varepsilon_1}),\ \ f_{23}\equiv f_{34}\mod c_1^T(L_{2\varepsilon_2}),\\
	&f_{12}\equiv f_{23}\equiv f_{25} \mod c_1^T(L_{2\varepsilon_1}),\ \ f_{14}\equiv f_{34}\equiv f_{45} \mod c_1^T(L_{2\varepsilon_1}),\\
	&f_{25}\equiv f_{35}\equiv f_{45} \mod c_1^T(L_{2\varepsilon_2}), \ \ f_{12}\equiv f_{13}\equiv f_{14} \mod c_1^T(L_{2\varepsilon_2}), \\ &(f_{12}-f_{23})+\rho_{1/2}c_1^T(L_{2\varepsilon_1})(f_{25}-f_{12})\equiv 0 \mod c_1^T(L_{2\varepsilon_1})^2,\\
	&(f_{14}-f_{34})+\rho_{1/2}c_1^T(L_{2\varepsilon_1})(f_{45}-f_{14})\equiv 0 \mod c_1^T(L_{2\varepsilon_1})^2,\\ &(f_{25}-f_{35})+\rho_{1/2}c_1^T(L_{2\varepsilon_2})(f_{45}-f_{25})\equiv 0 \mod c_1^T(L_{2\varepsilon_2})^2,\\ &(f_{12}-f_{13})+\rho_{1/2}c_1^T(L_{2\varepsilon_2})(f_{14}-f_{12})\equiv 0 \mod c_1^T(L_{2\varepsilon_2})^2.\\
	\end{align*}
	These equations give a complete description of the rational equivariant algebraic cobordism ring of $\IG(2,5)$. 
	
	Furthermore, we would like to identify the elements in the algebra given by the equations with geometric $T$-stable cobordism cycles in $\Omega_T^\ast(\IG(2,5))_\Q$. As an example, we consider the $T$-stable projective space $\PZ^2_{14,34,45}$ containing the fixed points $x_{14}, x_{34}$ and $x_{45}$ in $\IG(2,5)$. This leads to 
	\begin{align*}
	i_{x_{14}}^\ast[\PZ^2_{14,34,45}\to \IG(2,5)]&=n_1c_1^T(L_{\varepsilon_1-\varepsilon_2})c_1^T(L_{2\varepsilon_2})^2\\
	i_{x_{34}}^\ast[\PZ^2_{14,34,45}\to \IG(2,5)]&=n_2c_1^T(L_{\varepsilon_1-\varepsilon_2})c_1^T(L_{\varepsilon_1+\varepsilon_2})c_1^T(L_{2\varepsilon_2})\\
	i_{x_{45}}^\ast[\PZ^2_{14,34,45}\to \IG(2,5)]&=n_3c_1^T(L_{\varepsilon_1+\varepsilon_2})c_1^T(L_{2\varepsilon_2})^2
	\end{align*} 
	for some $n_1,n_2,n_3\in S(T)_\Q$ of degree zero. We do not know at this point which particular choice of the $n_i$ determines the pullback of the class $[\PZ^2 _{14,34,45}\to \IG(2,5)]$, but one of those tuples is certainly its image under the pullback map $i^\ast$. We will come back to this problem in the sequel of this article.   
\end{bei}
\section{Equivariant cobordism of horospherical varieties of Picard rank one}
In this section, we let $G$ be a connected reductive linear algebraic group, $B$ a fixed Borel subgroup with maximal torus $T$ and $W=N(T)/T$ the Weyl group. We want to compute the equivariant algebraic cobordism of smooth projective horospherical varieties of Picard number one. We begin this section describing the $T$-stable curves in flag varieties which will be important in order to describe the geometry of horospherical varieties. After that, we will define horospherical varieties and recall some of their basic notions as well as their geometry. Excellent references for the geometry of horospherical varieties are for example \cite{Perrin_Geometry_of_horospherical_varieties,Pasquier_horospherical_classification}. Using these descriptions, we will be able to describe the rational equivariant cobordism ring of horospherical varieties of Picard number one in terms of Theorem \ref{EquivCobThmSpherical}. 
\subsection{\texorpdfstring{$\boldsymbol{T}$}{T}-stable curves in flag varieties}\label{Tinvariant_curves_flags}
In this section, we will recall the main notions and results on $T$-stable curves in flag varieties $G/P$ from \cite[Section 3]{Fulton_Woodward}. We denote by $R=R^+\cup R^-$ the positive and negative roots and by $S$ the simple roots. Furthermore, we denote by $s_\alpha$ the reflections in $W$ which are indexed by positive roots $\alpha$. These are simple reflections if $\alpha$ is in $S$. For a subset $I\subseteq S$, let $W_I$ be the group which is generated by the reflections $s_\alpha$ for $\alpha$ in $I$. In addition, let $P_I=\coprod_{w\in W_I}BwB$ and $R^+_{P_I}$ be the set of positive roots that can be written as sums of roots in $I$. This is the well-known correspondence between parabolic subgroups $P_I$ of $G$ containing $B$ and subsets $I\subseteq S$. The length $\ell(w)$ of an $w\in W$ is the minimum number of simple reflections whose product is $w$. 

For any $u\in W/W_I$ we let $X(u)=\overline{BuP_I/P_I}$ be the corresponding Schubert variety which is of dimension $\ell(u)$ where $\ell(u)$ denotes the unique minimum length of a representative of $u$ in $W$. We denote its cohomology class $[X(u)]$ by $\sigma(u)$. Furthermore, for any $u\in W/W_I$ we denote by $x(u)=uP_I/P_I$ the corresponding $T$-fixed point in $G/P_I$. The Schubert classes of dimension one have the form $\sigma(s_\beta)$ as $\beta$ varies over $S\setminus I$. We define a degree $d$ to be a nonnegative integral combination $d=\sum d_\beta \sigma(s_\beta)$. The degrees are the classes of curves on $G/P_I$. For any positive root $\alpha$, we write $\alpha=\sum n_{\alpha\beta}\beta$ as the nonnegative sum of simple roots $\beta$. Then we define the degree $d(\alpha)$ of $\alpha$ by
\begin{align*}
d(\alpha):=\sum_{\beta\in S\setminus I}n_{\alpha\beta}\frac{(\beta,\beta)}{(\alpha,\alpha)}\sigma(s_\beta).
\end{align*}     
\begin{bem}\label{Degree_of_Curves}
	If $h_\alpha=2\alpha/(\alpha,\alpha)$ and $\omega_\beta$ is the fundamental weight corresponding to $\beta$, then $h_\alpha(\omega_\beta)=n_{\alpha\beta}(\beta,\beta)/(\alpha,\alpha)$ which implies 
	\begin{align*}
	d(\alpha)=\sum_{\beta\in S\setminus I}h_\alpha(\omega_\beta)\sigma(s_\beta).
	\end{align*}
\end{bem}
\begin{lem}\cite[Lemma 3.1]{Fulton_Woodward}
	If $w$ is in $W_I$, then we have $d(w(\alpha))=d(\alpha)$.
\end{lem}
For any positive root $\alpha$ which is not in $R^+_{P_I}$, there is a unique $T$-stable curve $C_\alpha$ in $G/P_I$ that contains the points $x(1)$ and $x(s_\alpha)$. We know that $C_\alpha=Z_\alpha\cdot P_I/P_I$ where $Z_\alpha$ is the 3-dimensional subgroup of $G$ whose Lie algebra is $\mathfrak{g_\alpha}\oplus\mathfrak{g_{-\alpha}}\oplus[\mathfrak{g_\alpha},\mathfrak{g_{-\alpha}}]$. 
\begin{lem}\cite[Lemma 3.4]{Fulton_Woodward}
	The degree $[C_\alpha]$ of $C_\alpha$ is $d(\alpha)$.
\end{lem}
\begin{defi}\label{adjacent}
	We say that two unequal elements $u$ and $v$ in $W/W_I$ are \textbf{adjacent} if there is a reflection $s_\alpha$ in $W$ for $\alpha \in R^+$ such that $v=s_\alpha u$. In this case we define $d(u,v)$ to be the degree $d(\alpha)$. 
\end{defi}
\begin{bem}
	If $u$ and $v$ are adjacent, then for any $w\in W$, the elements $wu$ and $wv$ are also adjacent and $d(wu,wv)=d(u,v)$ holds.
\end{bem}
\begin{lem}\cite[Lemma 4.2]{Fulton_Woodward}
	Elements $u$ and $v$ in $W/W_I$ are adjacent if and only if $x(u)\neq x(v)$ and there is a $T$-stable curve $C$ connecting $x(u)$ and $x(v)$. In this case, the curve $C$ is unique, isomorphic to $\PZ^1$ and its degree is equal to $d(u,v)$.
\end{lem}
\begin{bem}
	A general $T$-stable curve in $G/P_I$ has the form $w\cdot C_\alpha$ for some $\alpha\in R^+\setminus R^+_{P_I}$ and $w\in W$. This curve is the unique $T$-stable curve connecting $x(w)=w\cdot x(1)$ and $x(w s_\alpha)=w\cdot x(s_\alpha)$. 
\end{bem}
\begin{bei}\label{Example_G2_curves_closed orbits}
	We consider the flag variety $G_2/P_\alpha$ where $\alpha$ and $\beta$ denote the simple roots of $G_2$, $\beta$ being the long root. The flag variety $G_2/P_\alpha$ is a 5-dimensional quadric whose geometry was also studied in \cite{Perrin_Geometry_of_horospherical_varieties}. The positive roots are given by
	\begin{align*}
		R^+=\{\alpha,\beta,\alpha+\beta,2\alpha+\beta,3\alpha+\beta,3\alpha+2\beta\}.
	\end{align*} 
	Furthermore, we know $W_{\alpha}$ is generated by $s_\alpha$ which has order 2. Thus, we have 6 $T$-fixed points in $G_2/P_\alpha$ which are indexed by the elements of $W/W_\alpha$. From the above we know that for any $\gamma\in R^+\setminus R^+_{P_\alpha}$ there exists a unique $T$-stable curve connecting $x(1)$ and $x(s_\gamma)$. A short computation shows that we can find a reflection in $W$ for any two fixed elements in $W/W_\alpha$ such that they are adjacent. This implies that there is a $T$-stable curve connecting any two of the $T$-fixed points in $G_2/P_\alpha$ which leads to a total of 15 $T$-stable curves in the flag variety $G_2/P_\alpha$. Similar computations can be done for the flag variety $G_2/P_\beta$. 
\end{bei}
Later on, we will be interested in the weight acting on a $T$-stable curve $C$ and also its degree. To obtain those one can use the following lemma.
\begin{lem}\cite[Lemma 2.1]{Fulton_Woodward}
	Let a torus $T$ act on a curve $C\cong \PZ^1$ with two different $T$-fixed points $p$ and $q$ and let $L$ be a $T$-equivariant line bundle on $C$. Let $\chi_p$ and $\chi_q$ be the weights of $T$ acting on the fibers $L_p$ and $L_q$, respectively, and $\psi_p$ the weight of $T$ acting on the tangent space of $C$ at $p$. Then we have
	\begin{align*}
		\chi_p-\chi_q=n\psi_p
	\end{align*}
	where $n$ is the degree of $L$ on $C$.	
\end{lem}
\begin{bem}
	In our case, and more specifically in the previous Example \ref{Example_G2_curves_closed orbits}, the degree can also be obtained by Remark \ref{Degree_of_Curves}. These computations lead to 6 $T$-stable curves of degree 1, 6 $T$-stable curves of degree 3 and 3 $T$-stable curves of degree 2 in the flag variety $G_2/P_{\alpha}$ in Example \ref{Example_G2_curves_closed orbits}.	
\end{bem}    
\subsection{Geometry of horospherical varieties of Picard number one}
In this section, we focus on the class of \textbf{horospherical varieties} which is a special case of spherical varieties. We give two equivalent definitions and refer to \cite{Perrin_Geometry_of_horospherical_varieties,PhDThesis_Pasquier,horospherical,Pasquier_horospherical_classification} for more details on the geometry of horospherical varieties.
\begin{defi}
	Let $X$ be a normal $G$-variety.
	\begin{enumerate}[(i)]
		\item Let $H\subseteq G$ be a closed subgroup containing the unipotent radical $U$ of $B$. In this case, the homogeneous space $G/H$ is said to be \textbf{horospherical}.
		\item We call $X$ \textbf{horospherical} if it contains an open orbit isomorphic to a horospherical homogeneous space. 
	\end{enumerate}
\end{defi}
\begin{bem}
	A short computation shows that this open orbit isomorphic to $G/H$ contains an open orbit under the action of the Borel subgroup and therefore, a horospherical variety is spherical.
\end{bem} 
Now we give the second definition of horospherical varieties using a more geometric description. 
\begin{bem}
	A horospherical homogeneous space $G/H$ can be equivalently described as a torus bundle over a flag variety $G/P$ with fiber $P/H$. In this situation we have $P=N_G(H)$ by \cite[Proposition 2.2]{horospherical}. Furthermore, one has $P=TH=BH$ for all maximal tori $T$ of $B$ and all Borel subgroups contained in $P$. 
\end{bem}
\begin{defi}
	For a horospherical homogeneous space $G/H$, we call the dimension of the fiber $P/H$ the \textbf{rank of $\boldsymbol{G/H}$}. Furthermore, for a horospherical variety $X$, the \textbf{rank of $\boldsymbol{X}$} is defined as the rank of its open $G$-orbit.
\end{defi}
In this article, we focus on smooth projective horospherical varieties of Picard number one which have been classified by Pasquier \cite{Pasquier_horospherical_classification} in the following theorem.
\begin{thm}\cite[Theorem 0.1]{Pasquier_horospherical_classification}\label{Pasquier_classification}
	Let $G$ be a connected reductive algebraic group. Let $X$ be a smooth projective horospherical $G$-variety with Picard number one. Then one of the following cases can occur
	\begin{enumerate}[(i)]
		\item $X$ is homogeneous.
		\item $X$ is horospherical of rank 1. Its automorphism group is a connected non-reductive linear algebraic group, acting with exactly two orbits.
	\end{enumerate}
	Moreover, in the second case $X$ is uniquely determined by its two closed $G$-orbits $Y$ and $Z$, isomorphic to $G/P_Y$ and $G/P_Z$, respectively, and $(G,P_Y,P_Z)$ is one of the triples of the following list. 
	\begin{enumerate}[(1)]
		\item $(B_n,P(\omega_{n-1}),P(\omega_n))$ for $n\geq 3$
		\item $(B_3,P(\omega_1),P(\omega_3))$
		\item $(C_n,P(\omega_m),P(\omega_{m-1}))$ for $n\geq 2$ and $m\in [2,n]$
		\item $(F_4,P(\omega_2),P(\omega_3))$
		\item $(G_2,P(\omega_1),P(\omega_2)$ 
	\end{enumerate}	
	Here we denote by $P(\omega_i)$ the maximal parabolic subgroup of $G$ corresponding to the fundamental weight $\omega_i$ where we use the notations from Bourbaki \cite{Bourbaki_4-6}.
\end{thm}
\begin{bem}
	In our notation $P(\omega_i)$ will always be the maximal parabolic subgroup $P_{S\setminus \alpha_i}$ for the simple root $\alpha_i$ associated to the fundamental weight $\omega_i$.
\end{bem}
\begin{lem}\cite[Lemma 1.2]{Pasquier_horospherical_classification}
	Let $G/H$ be a horospherical homogeneous space. Up to isomorphism of varieties, there exists at most one smooth projective $G/H$-embedding with Picard number one.
\end{lem}
In the sequel, we will be only interested in the cases which are not homogeneous because the cobordism for homogeneous varieties has been studied before. Therefore we recall the construction from \cite[Section 1.3]{Perrin_Geometry_of_horospherical_varieties}. 

Let $X$ be a smooth projective horospherical but non homogeneous variety of Picard number one with associated triple $(G,P_Y,P_Z)$. In this case, we denote the previous triple also by $(G,P(\omega_Y),P(\omega_Z))$ for the corresponding fundamental weights $\omega_Y$ and $\omega_Z$. Furthermore, the dense orbit is given by $G/H=G\cdot [v_Y+v_Z]\subseteq \PZ(V_Y\oplus V_Z)$ where $V_Y$ and $V_Z$ are the irreducible $G$-representations with highest weights $\omega_Y$ and $\omega_Z$ and the corresponding highest weight vectors $v_Y$ and $v_Z$. We conclude by the construction that $P_Y$ and $P_Z$ are the stabilisers of $[v_Y]$ and $[v_Z]$ in $\PZ(V_Y)$ and $\PZ(V_Z)$ and that $Y$ and $Z$ are the $G$-orbits of $[v_Y]$ and $[v_Z]$ in $\PZ(V_Y)$ and $\PZ(V_Z)$, respectively.

Now, we will be analysing the $T$-stable curves and the fixed point subschemes $X^{T'}$ for some given $X$ in order to be able to use Theorem \ref{EquivCobThmSpherical} to obtain the rational equivariant cobordism of $X$. In the previous section, we have already seen how to determine the $T$-stable curves in the closed orbits $G/P_Y$ and $G/P_Z$ which are flag varieties. Next, we will analyse the $T$-stable curves meeting the dense open orbit $G/H$ for any smooth projective horospherical variety $X$ of Picard number one. We will use the diagram
\begin{equation}\label{Diagram_horospherical}
\begin{tikzcd}[]
& G/H \arrow[d,"\pi"] & \\
& G/(P_Y\cap P_Z) \arrow[dl,"p_Y",labels=above left] \arrow[dr,"p_Z",labels=above right] & \\
G/P_Y & & G/P_Z
\end{tikzcd}
\end{equation}
where $\pi$ is the corresponding $\C^\ast$-bundle. 
\begin{defi}
	Let $C$ be a $T$-stable irreducible curve in the dense open orbit $G/H$. Then we define $S:=\pi^{-1}(\pi(C))$ to be the preimage of $\pi(C)$.
\end{defi} 
\begin{lem}\label{Two_cases_for_S}
	Let $C$ be a $T$-stable irreducible curve in the dense open orbit $G/H$. Then $S$ is $T$-stable and given by one of the following cases.
	\begin{enumerate}[(i)]
		\item $S$ is the curve $C$ itself.
		\item $S$ is a surface containing $C$.
	\end{enumerate}
\end{lem}
\begin{proof}
	Let $C$ be a given $T$-stable irreducible curve in the dense open orbit $G/H$. Then, $\pi(C)$ is also $T$-stable. The following two cases can occur for $\pi(C)$. 
	\begin{enumerate}[(i)]
		\item $\pi(C)=\{\ast\}$ is a point. Without loss of generality, we can choose this point to be the $B$-fixed point in $G/P:=G/(P_Y\cap P_Z)$. The $B$-fixed point is $1\cdot P/P$ and therefore the closure $\overline{C}\subseteq X$ of the fiber $\pi^{-1}(1\cdot P/P)=C$ is the line joining the $B$-fixed points $1\cdot P_Y/P_Y=[v_Y]\in Y$ and $1\cdot P_Z/P_Z=[v_Z]\in Z$ because $\pi$ is a $\C^\ast$-bundle and $B$-fixed points are mapped to $B$-fixed points via the projections $p_Y$ and $p_Z$. The other lines will be obtained by the Weyl group action. Those lines are $T$-stable by assumption.
		\item $\pi(C)$ is a $T$-stable irreducible curve. Without loss of generality, we can choose $\pi(C)=Z_\alpha\cdot P/P$ for some positive root $\alpha$ which is not in $R^+_P$ where $Z_\alpha$ is the 3-dimensional subgroup of $G$ whose Lie algebra is $\mathfrak{g_\alpha}\oplus\mathfrak{g_{-\alpha}}\oplus[\mathfrak{g_\alpha},\mathfrak{g_{-\alpha}}]$. This curve joins the $B$-fixed point $x(1)$ and $x(s_\alpha)=s_\alpha\cdot P/P$. Then we obtain a two-dimensional surface $S:=\pi^{-1}(\pi(C))$ because $\pi$ is a $\C^\ast$-bundle. This surface $S$ contains $C$ and is $T$-stable. Indeed, for any $gH\in S$ we have
		\begin{align*}
			\pi(tgH)=t\pi(gH)\in \pi(C)=\pi\pi^{-1}\pi(C)=\pi(S).
		\end{align*}Therefore, $tgH\in \pi^{-1}\pi(C)=S$ holds which proves the claim of $S$ being $T$-stable. The other curves are obtained by the Weyl group action.
	\end{enumerate} 
\end{proof}
\begin{lem}
	Any surface in a connected component of $X^{T'}$ for a singular codimension one subtorus $T'=\Ker(\alpha)^0$ for some positive root $\alpha$ is of the form $\overline{S}\subseteq X$ for some $T$-stable curve $C$ and $S=\pi^{-1}(\pi(C))$.
\end{lem}
\begin{proof}
	Without loss of generality, we assume that $X^{T'}$ is connected. We know that $X^{T'}\cap G/H \neq \emptyset$ because the two closed orbits $Y\cong G/P_Y$ and $Z\cong G/P_Z$ contain only finitely many $T$-stable curves. Now let $x\in X^{T'}$. Then we have $tx=tt'x=t'tx$  for all $t,t'\in T$ which implies that $X^{T'}$ is $T$-stable. We have $X^{T'}\cap G/H\subseteq \pi^{-1}\pi(X^{T'}\cap G/H)$ and the reversed inclusion is also true because $X^{T'}$ is $T$-stable and $T$ acts transitively on the fibers of $\pi$ as $P/H$ is a quotient of $T$. Therefore, the whole fiber must be in $X^{T'}\cap G/H$. Furthermore, $X^{T'}$ has only zero- and one-dimensional $T$-orbits since $T/T'$ is one-dimensional. The image under $\pi$ of those orbits is either a $T$-fixed point or the $T$-stable irreducible curve $\pi(X^{T'}\cap G/H)$. We conclude that there must be a $T$-stable curve $C\subseteq X^{T'}\cap G/H$ such that $X^{T'}\cap G/H=\pi^{-1}\pi(C)$ because if the $T$-orbits were only the fibers then there would be infinitely many $T$-fixed points in $G/P$.
\end{proof}
\begin{bem}
	Let a connected component of $X^{T'}$ be given for some codimension one subtorus $T'$. As mentioned already in the previous proof these are $T$-stable with only zero- and one-dimensional $T$-orbits since $T/T'$ is one-dimensional. To be more precise, either an orbit is a $T$-fixed point or a one-dimensional $T$-orbit $T/T'\cdot x$ for some $x\in X^{T'}$. Therefore, the stabilisers of $x$ in $T$ are subtori of codimension one or zero.
\end{bem}
In the following, we want to analyse which surfaces $S$ are a connected component in some $X^{T'}$ for some codimension one subtorus $T'$. Therefore, we formulate the following lemma.
\begin{defi}
	Let $X$ be a smooth projective horospherical $G$-variety of Picard number one of the form $(G,P(\omega_Y),P(\omega_Z))$. Then we denote by $\chi:=\omega_Y-\omega_Z$ the difference of the two fundamental weights $\omega_Y$ and $\omega_Z$.
\end{defi}
\begin{lem}
	For any smooth projective horospherical variety $X$ of Picard number one we have the following properties.
	\begin{enumerate}[(1)]
		\item The only $T$-stable curves in $X$ meeting the open orbit $G/H$ occurring as a connected component of $X^{T'}$ for some codimension one subtorus $T'$ are of the form $\overline{\pi^{-1}(z)}$ where $z\in G/(P_Y\cap P_Z)$ is a $T$-fixed point.
		\item The surfaces occurring in $X^{T'}$ only arise from codimension one subtori of the form $T'=\Ker(w\alpha)^0=\Ker(w\chi)^0$ for some positive root $\alpha$ and some $w\in W$. 
	\end{enumerate}
\end{lem}
\begin{proof}
	As above, we have the $B$-fixed point $1\cdot P/P$ in $G/P:=G/(P_Y\cap P_Z)$. We need to consider the previously discussed case from Lemma \ref{Two_cases_for_S} (ii). Therefore, we assume that there exists a $T$-stable curve $C\subseteq G/H$ such that a general point in the $T$-stable curve $\pi(C)$ has the form $w\cdot u_{-\alpha}(x)\cdot P/P$ where $u_{-\alpha}(x)$ denotes the corresponding element in the root subgroup $U_{-\alpha}$. A general point in $S=\pi^{-1}\pi(C)$ has the form 
	\begin{align*}
	w\cdot u_{-\alpha}(x)tH=u_{-w\alpha}(x')wtH=u_{-w\alpha}(x')wtw^{-1}wH=u_{-w\alpha}(x')t'wH
	\end{align*}
	for $t\in P/H=\C^\ast$. Now we consider the $T$-action on those points for $z\in T$:
	\begin{align*}
	zu_{-w\alpha}(x')t'wH&=u_{-w\alpha}((w\alpha)(z)^{-1}x')zt'wH\\
	&=u_{-w\alpha}((w\alpha)(z)^{-1}x')t'zwH\\
	&=u_{-w\alpha}((w\alpha)(z)^{-1}x')t'ww^{-1}zwH.
	\end{align*} 
	This implies that a point $z$ acts trivially if and only if $w^{-1}zw\in H=\Ker(\chi)$ and $z\in \Ker(w\alpha)$ hold. This implies by the Weyl group action on the character group that this is equivalent to $z\in \Ker(w\chi)\cap \Ker(w\alpha)$.
	
	If $\Ker(w\chi)^0\neq \Ker(w\alpha)^0$ holds, then $\Ker(w\chi)^0\cap \Ker(w\alpha)^0$ has codimension two in $T$. Therefore, we obtain a $T$-stable surface $S$ or a $T$-stable curve in $\pi^{-1}\pi(C)$. It remains to check whether those are fixed by some codimension one subtorus. If one of those was a connected component of $X^{T'}$, then the stabiliser of any point in $X^{T'}$ would have at most codimension one in $T$, but as we computed above, the stabiliser of a general point in $S$ and therefore also in every potential $T$-stable curve in $\pi^{-1}\pi(C)$ is precisely $\Ker(w\chi)\cap \Ker(w\alpha)$. Therefore, the stabiliser of a general point would be of codimension two in $T$ and thus, the $T$-stable surface $S$ is not a connected component of $X^{T'}$ and there exists no $T$-stable curve in $\pi^{-1}\pi(C)$ which is a connected component of $X^{T'}$. 
	
	If $\Ker(w\alpha)=\Ker(w\chi)$ holds, then we have $\Ker(w\chi)^0=\Ker(w\alpha)^0=T'$ and $S$ is some connected component of $X^{T'}$ because $z$ acts trivially on a general point of $S$. This implies property (2). 
	
	Furthermore, there cannot be a connected component of $X^{T'}$ which is a $T$-stable curve in $G/H$ coming from Lemma \ref{Two_cases_for_S} (ii) because we obtain a surface as a connected component. This implies property (1). Thus, the only $T$-stable curves meeting the open orbit which are not contained in a connected component $S\subseteq X^{T'}$ of dimension two might be the lines described in Lemma \ref{Two_cases_for_S} (i) if they are not already contained in some $S$. 
\end{proof}
\textbf{Algorithm:} We analyse the occuring surfaces in $X^{T'}$. As we have seen above, we need to consider roots $\alpha$ which are multiples of the difference $\chi$ of the two fundamental weights $\omega_Y$ and $\omega_Z$ up to the Weyl group action. After that, we look at the curves in the closed orbits $Y$ and $Z$. Up to Weyl group action these are given by $Z_\alpha [v_{\omega_Y}]$ which connect $s_\alpha[v_{\omega_Y}]$ and $[v_{\omega_Y}]$ in $Y$ and similarly in $Z$. Thus, we need to compute $s_\alpha(\omega_Y)=\omega_Y-(\alpha^\vee,\omega_Y)\alpha$ and similarly for $\omega_Z$. Then we will know how many $T$-fixed points we have in $X^{T'}$ and in which orbits they occur. If we obtain 3 $T$-fixed points then we obtain a projective plane and if we obtain 4 $T$-fixed points, then we will have a Hirzebruch surface $\F_n$. We remark that $s_\alpha(\omega_Y)=\omega_Y$ holds if and only if $(\alpha^\vee,\omega_Y)$ vanishes.

We consider some examples of the classification of Pasquier which are given by triples $(G,P_Y,P_Z)$. We will study their geometry using the above algorithm and the classification of Bourbaki \cite{Bourbaki_4-6}. 
\begin{bei}\label{Classification_surfaces}
	In this example we will discuss three of the possible cases from Proposition \ref{Pasquier_classification}.
	\begin{enumerate}[(i)]
		\item Firstly, we consider type (1), i.e. $(B_n,P(\omega_{n-1}),P(\omega_n))$ for $n\geq 3$. The fundamental weights are given by  
		\begin{align*}
		\omega_{n-1}&=\varepsilon_1+...+\varepsilon_{n-1}\\&=\alpha_1+2\alpha_2+...+(n-2)\alpha_{n-2}+(n-1)(\alpha_{n-1}+\alpha_n)\text{ and }\\
		\omega_n&=1/2(\varepsilon_1+...+\varepsilon_n)\\&=1/2(\alpha_1+2\alpha_2+...+n\alpha_n)
		\end{align*}
		for $\alpha_i=\varepsilon_i-\varepsilon_{i+1}$, $1\leq i\leq n-1$, and $\alpha_n=\varepsilon_n$.
		Therefore, we have
		\begin{align*}
		\chi&=\omega_{n-1}-\omega_n\\&=1/2(\varepsilon_1+...+\varepsilon_{n-1}-\varepsilon_n)\\&=1/2(\alpha_1+2\alpha_2+...+(n-1)\alpha_{n-1}+(n-2)\alpha_n).
		\end{align*} 
		The positive roots are given by $\varepsilon_i$ for $1\leq i\leq n$ and $\varepsilon_i\pm \varepsilon_j$ for $1\leq i<j\leq n$. This implies that there will not be any surface in $X^{T'}$ because there is no root which is a multiple of $\chi$. 
		\item Secondly, we consider type (3), i.e. $(C_n,P(\omega_m),P(\omega_{m-1}))$ with integers $n\geq 2$ and $m\in [2,n]$. The fundamental weights are given by
		\begin{align*}
		\omega_i&=\varepsilon_1+...+\varepsilon_i\\&=\alpha_1+2\alpha_2+...+(i-1)\alpha_{i-1}+i(\alpha_i+\alpha_{i+1}+...+\alpha_{n-1}+\frac{1}{2}\alpha_n)
		\end{align*}
		for $1\leq i\leq n$ and $\alpha_i=\varepsilon_i-\varepsilon_{i+1}$, $1\leq i\leq n-1$, and $\alpha_n=2\varepsilon_n$. Therefore, we have 
		\begin{align*}
		\chi&=\omega_m-\omega_{m-1}\\&=\varepsilon_m\\&=\alpha_m+...+\alpha_{n-1}+\frac{1}{2}\alpha_n.
		\end{align*} 
		The positive roots are given by $\varepsilon_i\pm\varepsilon_j$ for $1\leq i<j\leq n$ and $2\varepsilon_i$ for $1\leq i\leq n$. Thus, there is a positive root which is a multiple of $\chi$ namely $\alpha:=2\varepsilon_m$. Consequently, we have 
		\begin{align*}
		\alpha^\vee=\frac{2\alpha}{(\alpha,\alpha)}=\frac{2\cdot 2\varepsilon_m}{(2\varepsilon_m,2\varepsilon_m)}=\varepsilon_m
		\end{align*}
		and therefore we obtain
		\begin{align*}
		(\alpha^\vee,\omega_m)=(\varepsilon_m,\varepsilon_1+...+\varepsilon_m)=1
		\end{align*}
		and
		\begin{align*}
		(\alpha^\vee,\omega_{m-1})=(\varepsilon_m,\varepsilon_1+...+\varepsilon_{m-1})=0.
		\end{align*}
		This implies that we have 3 $T$-fixed points and that we obtain a projective plane in $X^{T'}$. We recover the odd symplectic Grassmannian $\IG(m,2n+1)$ and thus, in particular Example \ref{ExampleIG(2,5)} in the case $m=n=2$.
		\item  Lastly, we consider type (5), i.e. the triple $(G_2,P(\omega_1),P(\omega_2))$. The fundamental weights are given by
		\begin{align*}
		\omega_1&=-\varepsilon_2+\varepsilon_3\\&=2\alpha_1+\alpha_2\text{ and }\\
		\omega_2&=-\varepsilon_1-\varepsilon_2+2\varepsilon_3\\&=3\alpha_1+2\alpha_2
		\end{align*} 
		for $\alpha_1=\varepsilon_1-\varepsilon_2$ and $\alpha_2=-2\varepsilon_1+\varepsilon_2+\varepsilon_3$. Therefore, we have 
		\begin{align*}
		\chi&=\omega_1-\omega_2\\&=\varepsilon_1-\varepsilon_3\\&=-\alpha_1-\alpha_2.
		\end{align*}
		The positive roots are given by $\alpha_1,\alpha_2, \alpha_1+\alpha_2,2\alpha_1+\alpha_2,3\alpha_1+\alpha_2$ and $3\alpha_1+2\alpha_2$. Thus, $\alpha:=-\chi$ is a positive root and consequently, we have 
		\begin{align*}
		\alpha^\vee&=\frac{2\alpha}{(\alpha,\alpha)}=\frac{2(\varepsilon_3-\varepsilon_1)}{(\varepsilon_3-\varepsilon_1,\varepsilon_3-\varepsilon_1)}=\varepsilon_3-\varepsilon_1.
		\end{align*}
		Therefore, we obtain 
		\begin{align*}
		(\alpha^\vee,\omega_1)=(\varepsilon_3-\varepsilon_1,-\varepsilon_2+\varepsilon_3)=1
		\end{align*}
		and
		\begin{align*}
		(\alpha^\vee,\omega_2)=(\varepsilon_3-\varepsilon_1,-\varepsilon_1-\varepsilon_2+2\varepsilon_3)=3.
		\end{align*}
		This implies that we have 4 $T$-fixed points and that we obtain a Hirzebruch surface $\F_3$ by Remark \ref{Degree_of_Curves} which ensures that $(\alpha^\vee,\omega_{1})$ and $(\alpha^\vee,\omega_2)$ give us the degrees of the curves in the two closed orbits $Y$ and $Z$, respectively.
	\end{enumerate}
\end{bei}  
After having described the $T$-stable structures on these smooth projective horospherical varieties of Picard number one, we can describe their equivariant algebraic cobordism rings. This will be done using Theorem \ref{EquivCobThmSpherical}.
\begin{bei}\label{Cobordism_classification_all_examples}
	Here, we will give the equivariant cobordism rings of the previous three cases. Therefore, we will in general consider as usual the injective map
	\begin{align*}
	i^\ast:\Omega^\ast_T(X)_\Q\to \Omega^\ast_T(X^T)_\Q.
	\end{align*}
	\begin{enumerate}[(i)]
		\item At first, we consider the case $(B_n,P(\omega_{n-1}),P(\omega_n))$ for $n\geq 3$. For any element $w'\in W/W_{S\setminus \alpha_{n-1}}$ we denote by $y(w'):=w'P(\omega_{n-1})/P(\omega_{n-1})$ the corresponding $T$-fixed point in $Y$ and similarly by $z(w''):=w''P(\omega_n)/P(\omega_n)$ the $T$-fixed point in the closed orbit $Z$ for any $w''\in W/W_{S\setminus \alpha_n}$. The equations for the closed orbits $Y$ and $Z$ are given by 
		\begin{align}\label{Equations_in_Y}
		& &f_{y(w\cdot s_\alpha)}&\equiv f_{y(w)}&&\mod c_1^T(L_{w\omega_{n-1}-ws_\alpha\omega_{n-1}})\\\label{Equations_in_Z}
		& & f_{z(w\cdot s_\beta)}&\equiv f_{z(w)} &&\mod c_1^T(L_{w\omega_n-ws_\beta\omega_n})
		\end{align}
		for $\alpha\in R^+\setminus R^+_{P(\omega_{n-1})},\beta\in R^+\setminus R^+_{P(\omega_{n})}$ and $w\in W$ which is true as the difference of the weights associated to the $T$-fixed points is a multiple of the weight acting on the corresponding curve and we consider rational coefficients. We have seen above that there are no surfaces in this particular case. Therefore, the last equations are given by the lines joining the two closed orbits. These are given by
		\begin{align}\label{Equations_lines_dense_orbit}
		f_{y(w)}\equiv f_{z(w)} \mod c_1^T(L_{w\omega_{n-1}-w\omega_n})
		\end{align}
		for $w\in W$. This describes completely the equivariant algebraic cobordism $\Omega^\ast_T(X)_\Q$ in case (1). 
		\item Secondly, we consider the case $(C_n,P(\omega_m),P(\omega_{m-1}))$ for $n\geq 2$ and $m\in [2,n]$. The equations for the curves in the closed orbits can be obtained as in (\ref{Equations_in_Y}) and (\ref{Equations_in_Z}). Furthermore, the equations from the lines joining the closed orbits can be obtained as in (\ref{Equations_lines_dense_orbit}). As we have seen in Example \ref{Classification_surfaces}, we need to choose $\alpha:=2\varepsilon_m$ to be the positive root which is a multiple of $\chi=\omega_m-\omega_{m-1}$ in order to obtain a surface in $X^{T'}$ for $T'=\Ker(\alpha)^0$. The reflection $s_\alpha$ acts trivially on the $T$-fixed point $z(1)$ and therefore, we obtain the $T$-fixed points $z(1),y(1)$ and $y(s_\alpha)$ in $X^{T'}$. Having a look at the weights acting on the lines in the resulting surface $\PZ^2$, we can identify the $T$-fixed points $z(1),y(1)$ and $y(s_\alpha)$ with $y,x$ and $z$, respectively, where we consider the canonical $T$-action on $\PZ^2$, i.e. $t\cdot [x:y:z]=[tx:y:t^{-1}z]$. For any $w\in W$ this leads to the equation
		\begin{align*}
		(f_{y(w)}-f_{z(w)})+\rho_{1/2}c_1^T(L_{w\alpha})(f_{y(w\cdot s_\alpha)}-f_{y(w)})\equiv 0 \mod c_1^T(L_{w\alpha})^2.
		\end{align*} 
		This completes the description of the equivariant algebraic cobordism in case (3). Furthermore, we remark that we recover precisely the description of the rational equivariant algebraic cobordism of $\IG(2,5)$ from Example \ref{ExampleIG(2,5)} for $m=n=2$.
		\item Lastly, we consider case (5) which is given by the triple $(G_2,P(\omega_1),P(\omega_2))$ for $\omega_1=2\alpha_1+\alpha_2$ and $\omega_2=3\alpha_1+2\alpha_2$. The curves can be described as above for the previous cases. In order to obtain surfaces in $X^{T'}$ we need to choose $\alpha:=-\chi$ by Example \ref{Classification_surfaces}. Therefore, we obtain the $T$-fixed points $y(1),y(s_\alpha),z(1)$ and $z(s_\alpha)$ contained in a Hirzebruch surface $\F_3$ which has been described in Example \ref{Classification_surfaces}. By that example we know that we have a curve of degree 1 in $Y$ and one of degree 3 in $Z$. By verifying the weights we can identify $y(1),y(s_\alpha),z(1)$ and $z(s_\alpha)$ with $x,y,w$ and $z$, respectively, using the notion from Proposition \ref{Brion 7.2.Cob}. For any $w'\in W$ we define $\xi_{w'\cdot s_\alpha}:=(f_{y(w'\cdot s_\alpha)}-f_{z(w'\cdot s_\alpha)})$ and $\xi_{w'}:=(f_{z(w')}-f_{y(w')})$ which leads to the equations
		\begin{align*}
		\rho_{3/2}c_1^T(L_{w'\alpha})\xi_{w'\cdot s_\alpha}+\rho_{-3/2}c_1^T(L_{w'\alpha})\xi_{w'}\equiv 0\mod c_1^T(L_{w'\alpha})^2.
		\end{align*}  
		This completes the description of $\Omega_T^\ast(X)_\Q$ in case (5).
	\end{enumerate} 
\end{bei}
\begin{bem}
	To finish this section, we remark that the computations for the equivariant cobordism of the odd symplectic Grassmannian $\IG(2,5)$ with the geometric description from Example \ref{ExampleIG(2,5)} can be generalised to all the examples of type (3), i.e. to all odd symplectic Grassmannians $\IG(m,2n+1)$ for $n\geq 2$ and $m\in [2,n]$. 
\end{bem}
\section{Equivariant multiplicities at nondegenerate fixed point in cobordism}\label{Section_equiv_Multiplicities}
In this section, we want to generalise some results for equivariant Chow groups from \cite[Section 4]{BrionTorusActions} to equivariant algebraic cobordism. 
\begin{defi}\label{non_degenerate_fixed_points}
	Let $X$ be a scheme with a $T$-action. We call a $T$-fixed point $x\in X$ \textbf{nondegenerate} if the tangent space $T_xX$ contains no nonzero fixed point. Equivalently, $0$ is not a weight for the $T$-module $T_xX$. The weights of this module counted with their equivariant multiplicities will be called the \textbf{weights of $\boldsymbol{x}$ in $\boldsymbol{X}$}. 
\end{defi}
\begin{bem}\cite[Section 4.1]{BrionTorusActions}
	We have $T_x(X^T)=(T_xX)_0$ where $(T_xX)_0$ denotes the sum of the weight subspaces of $T_xX$ with zero weight. Therefore, any $T$-fixed point in a nonsingular $T$-variety is nondegenerate if and only if it is isolated. Thus, for the class of smooth projective and spherical varieties all $T$-fixed points are nondegenerate.   	
\end{bem}
Before we start to prove the main analogues of \cite[Section 4]{BrionTorusActions} we recall two important statements which were proved by Krishna \cite{CobTorus}. Recall that $S(T)[M^{-1}]$ is the graded ring obtained by inverting all non-zero linear forms $\sum_{j=1}^{n}m_jt_j$ which was described in more detail in \cite[Section 6]{CobTorus}. For a smooth $k$-scheme $X$ with a torus action, we denote $\Omega_T^\ast(X)\otimes_{S(T)}S(T)[M^{-1}]$ by $\Omega_T^\ast(X)[M^{-1}]$. 
\begin{thm}\cite[Proposition 3.1]{CobTorus}\label{Self-intersection_formula}
	Let $G$ be a linear algebraic group and $f:Y\to X$ be a regular $G$-equivariant embedding in $G-\boldsymbol{\Sch}_k$ of pure codimension $d$ and let $N_{Y/X}$ denote the equivariant normal bundle of $Y$ inside $X$. Then one has
	\begin{align*}
		f^\ast\circ f_\ast(\eta)=c_d^G(N_{Y/X})(\eta)
	\end{align*}
	for every $\eta\in \Omega^G_\ast(Y)$.
\end{thm}
\begin{kor}\cite[Corollary 7.3]{CobTorus}\label{Krishna_Corollary_7.3}
	Let $X$ be a smooth projective variety with an action of a torus $T$ of rank $n$. Then the pushforward map $i_\ast: \Omega_\ast^T(X^T)\to \Omega_\ast^T(X)$ becomes an isomorphism after base change to $S(T)[M^{-1}]$.
\end{kor}
We recall that the equivariant cobordism module of disconnected varieties is the sum of the equivariant cobordism modules of the connected components. 
\begin{defi}\label{equivariant_mulitplicities}
	Let $X$ be a smooth projective variety with an action of a torus $T$. Further, let $[Y\to X]\in \Omega_\ast^T(X)[M^{-1}]$ and $x\in X$ be an isolated $T$-fixed point. We distinguish between isolated fixed points and connected components $F\subseteq X^T$ which are not an isolated point. For any isolated fixed point we define the \textbf{equivariant multiplicity} $e_{x,X}[Y\to X]\in S(T)[M^{-1}]$ of $X$ at $x$ to be given by the equality 
	\begin{align*}
		[Y\to X]=i_\ast\left(\sum_{\substack{x\in X^T\\ \text{isolated}}}e_{x,X}[Y\to X][x\to x]+\sum_{F\subseteq X^T}e_F[F'\to F]\right)
	\end{align*}
	which holds in $\Omega_\ast^T(X)[M^{-1}]$ for some $e_F\in S(T)[M^{-1}]$ and $[F'\to F]\in \Omega_\ast^T(F)$.
\end{defi}
\begin{lem}\label{Smooth_classes}
	Let $X$ be a smooth projective scheme with a $T$-action. Furthermore, let $Y\subseteq X$ be a closed smooth subscheme. For the class $[f:Y\to X]$ in the $S(T)$-algebra $\Omega_T^\ast(X)$ and any nondegenerate $T$-fixed point $y\in Y$ we have 
	\begin{align*}
		e_{y,X}[Y\to X]=\frac{1}{c_1^T(L_{-\chi_1})\cdots c_1^T(L_{-\chi_{m}})}
	\end{align*}
	in $\Omega^T_\ast(X)[M^{-1}]$ where $\chi_1,...,\chi_{m}$ are the weights of $y$ in $Y$. 
\end{lem}     
\begin{proof}
	First, we consider the equality
	\begin{align}\label{X_m_multiplicity}
		[Y\to Y]=\sum_{\substack{y\in Y^T\\ \text{isolated}}}e_{y,Y}[Y\to Y][y\to Y]+\sum_{F\subseteq Y^T}e_F[F'\to Y].
	\end{align}
	coming from Definition \ref{equivariant_mulitplicities}.
	For $j:Y^T\to Y$ we apply $j^\ast$ on both sides. Using Proposition \ref{Self-intersection_formula} and the Whitney sum formula we obtain
	\begin{align*}
		[Y^T\to Y^T]=\sum_{\substack{y\in Y^T\\ \text{isolated}}}e_{y,Y}[Y\to Y]\left(\prod_{\substack{\chi \text{ weights of }\\y \text{ in }Y}}c_1^T(L_{-\chi})\right)[y\to y]+\sum_{F\subseteq Y^T}e_F[j^\ast F'\to F]	
	\end{align*}
	which leads to 
	\begin{align*}
		e_{y,Y}[Y\to Y]=\left(\prod_{\substack{\chi \text{ weights of }\\y \text{ in }Y}}c_1^T(L_{-\chi})\right)^{-1}
	\end{align*}
	for all $y\in Y^T$.
	Now, we apply $f_\ast$ to (\ref{X_m_multiplicity})and thus, we have
	\begin{align}\label{pushforward_equality}
		[Y\to X]=\sum_{\substack{y\in Y^T\\ \text{isolated}}}e_{y,Y}[Y\to Y][y\to X]+\sum_{F\subseteq Y^T}e_F[F'\to X].
	\end{align}
	On the other hand, by Definition \ref{equivariant_mulitplicities} we have the equality
	\begin{align*}
		[Y\to X]=\sum_{\substack{x\in X^T\\ \text{isolated}}}e_{x,X}[Y\to X][x\to X]+\sum_{\widetilde{F}\subseteq X^T}e_{\widetilde{F}}[\widetilde{F}'\to X].
	\end{align*}
	Let $i:X^T\to X$ be the inclusion of the fixed point subscheme of $X$. Applying $i^\ast$ implies $e_{x,X}[Y\to X]=0$ for all isolated fixed points $x\notin Y^T$. Similarly, $e_{\widetilde{F}}=0$ if $\widetilde{F}\nsubseteq Y^T$. Thus, we obtain 
	\begin{align}\label{Definition_equality}
		[Y\to X]=\sum_{\substack{y\in Y^T\\ \text{isolated}}}e_{y,X}[Y\to X][y\to X]+\sum_{\widetilde{F}\subseteq Y^T}e_{\widetilde{F}}[\widetilde{F}'\to X].
	\end{align}
	Pulling back the right-hand sides of (\ref{pushforward_equality}) and (\ref{Definition_equality}) along $i$ leads to decompositions in $\Omega_T^\ast(X^T)_\Q$ due to the fact that all images of morphisms $F'\to X$ and $\widetilde{F}'\to X$ are contained in different connected components $F$ and $\widetilde{F}$ of $Y^T$ and thus of $X^T$. Those components $F$ and $\widetilde{F}$ are disjoint from the set of isolated fixed points by assumption. This implies that the second sum cannot contribute to the classes $[y\to X]$ for isolated fixed points $y\in Y^T$. Hence, comparing coefficients in (\ref{pushforward_equality}) and (\ref{Definition_equality}) leads to
	\begin{align*}
		e_{y,Y}[Y\to Y]=e_{y,X}[Y\to X]
	\end{align*}
	which implies the claim.
\end{proof}
Next, we consider classes $[Y\to X]$ of the $S(T)$-algebra $\Omega_T^\ast(X)$ for which $Y$ is not necessarily a closed smooth subscheme of $X$.  
\begin{thm}\label{Non-smooth_classes}
	Let $X$ be a smooth projective scheme with a $T$-action. Let $x\in X$ be a nondegenerate fixed point and $[f:Y\to X]$ a class in the $S(T)$-algebra $\Omega_T^\ast(X)$. Assume further that all fixed points in the fiber $f^{-1}(x)$ are nondegenerate. Then we have
	\begin{align*}
		e_{x,X}[Y\to X]=\sum_{\substack{y\in Y^T\\
				f(y)=x}}e_{y,Y}[Y\to Y].
	\end{align*}
\end{thm}
\begin{proof}
	Let $j:U\to X$ be the inclusion of some open $T$-stable neighbourhood of $x$. By potential shrinking we may assume that $x$ is the unique $T$-fixed point in $X$. Using Definition \ref{equivariant_mulitplicities} in $\Omega_T^\ast(X)$ we obtain
	\begin{align*}
		[Y\to X]=\sum_{\substack{x\in X^T\\ \text{isolated}}}e_{x,X}[Y\to X][x\to X]+\sum_{\widetilde{F}\subseteq X^T}e_{\widetilde{F}}[\widetilde{F}'\to X].
	\end{align*}
	We have $j^\ast[\widetilde{F}'\to X]=0$ if $\Im(\widetilde{F}')\subseteq X$ does not contain $x$. Therefore, pulling back along $j$ yields
	\begin{align*}
		[f^{-1}(U)\to U]=\sum_{x\in U^T}e_{x,X}[Y\to X][x\to U]=e_{x,X}[Y\to X][x\to U].
	\end{align*}
	On the other hand, we have 
	\begin{align*}
		[Y\to Y]=i_\ast\left(\sum_{\substack{y\in Y^T\\ \text{isolated}}}e_{y,Y}[Y\to Y][y\to y]+\sum_{F\subseteq Y^T}e_F[F'\to F]\right).
	\end{align*}
	Applying the pushforward $f_\ast$ to the equation results in
	\begin{align*}
		[Y\to X]=\left(\sum_{\substack{y\in Y^T\\ \text{isolated}}}e_{y,Y}[Y\to Y][y\to X]+\sum_{F\subseteq Y^T}e_F[F'\to X]\right).
	\end{align*}
	Again, $j^\ast[F'\to X]=0$ and $j^\ast[y'\to X]=0$ for any $y'\in Y^T$ if $f(y')\neq x$. Thus, applying the pullback $j^\ast$ yields
	\begin{align*}
		[f^{-1}(U)\to U]=\sum_{\substack{y\in Y^T\\ f(y)=x}}e_{y,Y}[Y\to Y][y\to U].
	\end{align*} 
	Due to the fact that $[x\to U]=[y\to U]$ holds in $\Omega_T^\ast(U)[M^{-1}]$ for any $y\in Y^T$ with $f(y)=x$, we obtain
	\begin{align*}
		e_{x,X}[Y\to X][x\to U]&=\sum_{\substack{y\in Y^T\\ 				f(y)=x}}e_{y,Y}[Y\to Y][y\to U]\\&=\left(\sum_{\substack{y\in Y^T\\ 				f(y)=x}}e_{y,Y}[Y\to Y]\right)[y\to U].
	\end{align*}
	Thus, the corresponding coefficients in $S(T)[M^{-1}]$ must coincide which implies the claim.
\end{proof}
\begin{bei}
	We want to determine the corresponding classes of $\IG(2,5)$ in $S(T)^8_\Q$, see Example \ref{ExampleIG(2,5)}. Therefore, we consider the Bialynicki-Birula decomposition given by the generic one-parameter subgroup $t\mapsto \diag (t^2,t,t^{-1},t^{-2})$ coming from Brion's definition of $T$-filtrable varieties in \cite[Section 3]{BrionTorusActions}. Using Definition \ref{equivariant_mulitplicities} and Lemma \ref{Smooth_classes}, one can compute the pullbacks of the fixed points which are given by
	\begin{align*}
		i_{x_{45}}^\ast[x_{45}\to \IG(2,5)]&=c_1^T(L_{-\varepsilon_1-\varepsilon_2})c_1^T(L_{-2\varepsilon_2})c_1^T(L_{-\varepsilon_2})c_1^T(L_{-2\varepsilon_1})c_1^T(L_{-\varepsilon_1})\\
		i_{x_{35}}^\ast[x_{35}\to \IG(2,5)]&=c_1^T(L_{-\varepsilon_2})c_1^T(L_{\varepsilon_2})c_1^T(L_{-2\varepsilon_1})c_1^T(L_{-\varepsilon_1-\varepsilon_2})c_1^T(L_{\varepsilon_2-\varepsilon_1})\\
		i_{x_{34}}^\ast[x_{34}\to \IG(2,5)]&=c_1^T(L_{-\varepsilon_1})c_1^T(L_{\varepsilon_1})c_1^T(L_{-\varepsilon_1-\varepsilon_2})c_1^T(L_{-2\varepsilon_2})c_1^T(L_{\varepsilon_1-\varepsilon_2})\\
		i_{x_{25}}^\ast[x_{25}\to \IG(2,5)]&=c_1^T(L_{\varepsilon_2})c_1^T(L_{2\varepsilon_2})c_1^T(L_{-2\varepsilon_1})c_1^T(L_{-\varepsilon_1})c_1^T(L_{\varepsilon_2-\varepsilon_1})\\
		i_{x_{23}}^\ast[x_{23}\to \IG(2,5)]&=c_1^T(L_{\varepsilon_2-\varepsilon_1})c_1^T(L_{2\varepsilon_2})c_1^T(L_{\varepsilon_1+\varepsilon_2})c_1^T(L_{-\varepsilon_1})c_1^T(L_{\varepsilon_1})\\
		i_{x_{14}}^\ast[x_{14}\to \IG(2,5)]&=c_1^T(L_{\varepsilon_1-\varepsilon_2})c_1^T(L_{\varepsilon_1})c_1^T(L_{2\varepsilon_1})c_1^T(L_{-2\varepsilon_2})c_1^T(L_{-\varepsilon_2})\\
		i_{x_{13}}^\ast[x_{13}\to \IG(2,5)]&=c_1^T(L_{\varepsilon_1-\varepsilon_2})c_1^T(L_{\varepsilon_1+\varepsilon_2})c_1^T(L_{2\varepsilon_1})c_1^T(L_{-\varepsilon_2})c_1^T(L_{\varepsilon_2})\\
		i_{x_{12}}^\ast[x_{12}\to \IG(2,5)]&=c_1^T(L_{\varepsilon_1})c_1^T(L_{\varepsilon_1+\varepsilon_2})c_1^T(L_{2\varepsilon_1})c_1^T(L_{\varepsilon_2})c_1^T(L_{2\varepsilon_2})
	\end{align*}
	where $x_{45}$ is the most attractive fixed point and $\varepsilon_1,\varepsilon_2$ are given as in Example \ref{ExampleIG(2,5)}. Lastly, using Lemma \ref{Smooth_classes} and by computing the weights on stable neighbourhoods of the fixed points we deduce 
	\begin{align*}
		i_{x_{12}}^\ast[X_0\to \IG(2,5)]&=c_1^T(L_{\varepsilon_1})c_1^T(L_{\varepsilon_1+\varepsilon_2})c_1^T(L_{2\varepsilon_1})c_1^T(L_{\varepsilon_2})c_1^T(L_{2\varepsilon_2})\\
		i_{x_{12}}^\ast[X_1\to \IG(2,5)]&=c_1^T(L_{\varepsilon_1})c_1^T(L_{\varepsilon_1+\varepsilon_2})c_1^T(L_{2\varepsilon_1})c_1^T(L_{2\varepsilon_2})\\
		i_{x_{13}}^\ast[X_1\to \IG(2,5)]&=c_1^T(L_{\varepsilon_1-\varepsilon_2})c_1^T(L_{\varepsilon_1+\varepsilon_2})c_1^T(L_{2\varepsilon_1})c_1^T(L_{\varepsilon_2})\\
		i_{x_{12}}^\ast[X_2\to \IG(2,5)]&=c_1^T(L_{\varepsilon_1})c_1^T(L_{\varepsilon_1+\varepsilon_2})c_1^T(L_{2\varepsilon_1})\\
		i_{x_{13}}^\ast[X_2\to \IG(2,5)]&=c_1^T(L_{\varepsilon_1-\varepsilon_2})c_1^T(L_{\varepsilon_1+\varepsilon_2})c_1^T(L_{2\varepsilon_1})\\
		i_{x_{14}}^\ast[X_2\to \IG(2,5)]&=c_1^T(L_{\varepsilon_1-\varepsilon_2})c_1^T(L_{\varepsilon_1})c_1^T(L_{2\varepsilon_1})\\
		i_{x_{12}}^\ast[X_2'\to \IG(2,5)]&=c_1^T(L_{\varepsilon_1+\varepsilon_2})c_1^T(L_{2\varepsilon_1})c_1^T(L_{2\varepsilon_2})\\
		i_{x_{13}}^\ast[X_2'\to \IG(2,5)]&=c_1^T(L_{\varepsilon_1+\varepsilon_2})c_1^T(L_{2\varepsilon_1})c_1^T(L_{\varepsilon_2})\\
		i_{x_{23}}^\ast[X_2'\to \IG(2,5)]&=c_1^T(L_{\varepsilon_1+\varepsilon_2})c_1^T(L_{\varepsilon_1})c_1^T(L_{2\varepsilon_2})
	\end{align*}
	where $X_2$ and $X_2'$ are the two projective planes obtained by attaching one of the two affine planes to the projective line $X_1$. Therefore, the pullback $i^\ast[\widetilde{X}_3\to \IG(2,5)]$ is given by the sum of $i^\ast[X_2\to \IG(2,5)]$ and $i^\ast[X_2'\to \IG(2,5)]$ where $\widetilde{X}_3$ is the normalisation of $X_3$. 
	
	In the sequel, we set $E_i$ to be the vector space generated by the first $i$ basis vectors of $\C^5$. For the sake of completeness, we remark that $X_0,X_1,X_2$ and $X_2'$ are given by 
	\begin{align*}
		X_0&=\{x_{12}\}, \\
		X_1&=\{V_2\in \IG(2,5)\ \vert \ E_1\subseteq V_2\},\\
		X_2&=\{V_2\in \IG(2,5)\ \vert \ E_1\subseteq V_2\subseteq E_4\},\\
		X_2'&=\{V_2\in \IG(2,5)\ \vert \ V_2\subseteq E_3\}.
	\end{align*} 
	
	Now, we will consider the singular subscheme $X_4\subseteq \IG(2,5)$ which is obtained by attaching the $\A^3$ containing the fixed point $x_{25}$. Geometrically, $X_4$ can be identified with a cone over a surface with only one singular point $x_{12}$. The pullback to smooth $T$-fixed points in $X_4$ works similar as in the previous cases. Therefore, we only consider the pullback to the singular fixed point $x_{12}$. One can compute the blow up of the point $x_{12}$ in $X_4$ explicitly and check that there are four $T$-fixed points in the exceptional divisor $E$. Using Proposition \ref{Non-smooth_classes}, we need to compute the weights of the four $T$-fixed points in $E\subseteq\widetilde{X}_4$. These weights can be seen from the computation directly. Using Proposition \ref{Non-smooth_classes} and Definition \ref{equivariant_mulitplicities} leads to 
	\begin{align*}
		i^\ast_{x_{12}}[\widetilde{X}_4\to \IG(2,5)]&=e_{x_{12},\IG(2,5)}[\widetilde{X}_4\to \IG(2,5)]i^\ast_{x_{12}}[x_{12}\to \IG(2,5)]\\&=\left(\sum_{\substack{\tilde{x}\in \widetilde{X}_4^T\\ f(\tilde{x})=x_{12}}}e_{\tilde{x},\widetilde{X}_4}[\widetilde{X}_4\to \widetilde{X}_4]\right)i^\ast_{x_{12}}[x_{12}\to \IG(2,5)]\\ 
		&=\frac{c_1^T(L_{\varepsilon_1})c_1^T(L_{\varepsilon_1+\varepsilon_2})c_1^T(L_{\varepsilon_2})c_1^T(L_{2\varepsilon_2})}{c_1^T(L_{-\varepsilon_1})c_1^T(L_{\varepsilon_2-\varepsilon_1})}+\frac{c_1^T(L_{\varepsilon_1+\varepsilon_2})c_1^T(L_{2\varepsilon_1})c_1^T(L_{\varepsilon_2})c_1^T(L_{2\varepsilon_2})}{c_1^T(L_{\varepsilon_1})c_1^T(L_{\varepsilon_2-\varepsilon_1})}\\
		&+\frac{c_1^T(L_{\varepsilon_1})c_1^T(L_{\varepsilon_1+\varepsilon_2})c_1^T(L_{2\varepsilon_1})c_1^T(L_{\varepsilon_2})}{c_1^T(L_{-\varepsilon_2})c_1^T(L_{\varepsilon_1-\varepsilon_2})}+\frac{c_1^T(L_{\varepsilon_1})c_1^T(L_{\varepsilon_1+\varepsilon_2})c_1^T(L_{2\varepsilon_1})c_1^T(L_{2\varepsilon_2})}{c_1^T(L_{\varepsilon_2})c_1^T(L_{\varepsilon_1-\varepsilon_2})}.
	\end{align*}  
	We remark that this element reduces to the correct one in Chow rings and that the pullback $i^\ast_{x_{12}}[\widetilde{X}_4\to \IG(2,5)]$ is an element in $S(T)_\Q$. Alternatively, one could check that the geometric descriptions of $X_4$ and $\widetilde{X}_4$ are given by 
	\begin{align*}
		X_4&=\{V_2\in \IG(2,5)\ \vert \ E_2\cap V_2\neq 0\} \text{ and }\\
		\widetilde{X}_4&=\{(V_1,V_2,V_3)\in \PZ(\C^5)\times \IG(2,5)\times \Gr(3,5)\ \vert \ V_1\subseteq E_2\subseteq V_3\subseteq V_1^\perp,V_1\subseteq V_2\subseteq V_3\}.
	\end{align*}
	
	We consider now the closed subscheme $X_5\subseteq \IG(2,5)$ which is obtained by attaching the cell containing the fixed point $x_{34}$ to $X_4$. A short computation shows that the planes containing $x_{12},x_{13}, x_{14}$ and $x_{12},x_{13},x_{23}$ are singular in $X_5$. Normalising yields $X_4$ and $X_4':=X_3\cup (X_5\setminus X_4)$. We remark that $X_4'$ is given by the equations $e_4\land e_5=e_3\land e_5=e_2\land e_5=0$ which implies 
	\begin{align*}
		X_4'=\{V_2\subseteq \C^5 \text{ isotropic} \ \vert \ V_2\subseteq E_4\}.
	\end{align*}  
	One may observe that any isotropic subspace $V_2$ in $E_4$ has to remain isotropic when considering $\overline{V_2}:=(V_2+E_4^\perp)/E_4^\perp\subseteq E_4/E_4^\perp$, but since $E_4/E_4^\perp=\langle e_2,e_4\rangle$ holds, we obtain
	\begin{align*}
		X_4'=\{V_2\subseteq E_4\ \vert \ V_2\cap \langle e_1,e_3\rangle\neq 0\}.
	\end{align*} 
	We claim that a resolution $\widetilde{X}_4'$ of $X_4'$ is given by
	\begin{align*}
		\widetilde{X}_4'=\{(V_1,V_2,V_3)\in \PZ(E_4)\times X_4'\times \Gr(3,E_4)\ \vert \ V_1\subseteq V_2\cap \langle e_1,e_3\rangle, V_3\supseteq V_2+\langle e_1,e_3\rangle\}.
	\end{align*}
	This is birational to $X_4'$. Now, we consider the map
	\begin{align*}
		h:\widetilde{X}_4'\to \{(V_1,V_3)\ \vert \ V_1\subseteq \langle e_1,e_3\rangle, V_3\supseteq \langle e_1,e_3\rangle\}=\PZ^1\times\PZ^1
	\end{align*}
	which is a $\PZ^1$-fibration over $\PZ^1\times \PZ^1$. Therefore, $\widetilde{X}_4'$ is smooth and projective. The only singular point in $X_4'$ is $x_{13}$ and thus, we want to compute $i_{x_{13}}^\ast[\widetilde{X}_4'\to \IG(2,5)]$ using Proposition \ref{Non-smooth_classes}. The $T$-fixed points in the exceptional divisor are given by
	\begin{align*}
		(E_1,\langle e_1,e_3\rangle,E_3), (E_1,\langle e_1,e_3\rangle,\langle e_1,e_3,e_4\rangle), (e_3,\langle e_1,e_3\rangle,E_3) \text{ and } (e_3,\langle e_1,e_3\rangle,\langle e_1,e_3,e_4\rangle).
	\end{align*} 
	Exemplary, we compute the weights for the first $T$-fixed point in the exceptional divisor, i.e. for $\widetilde{x_1}:=(E_1,\langle e_1,e_3\rangle,E_3)$. Therefore, we consider the morphism $h$ and the tangent space $T_{h(\widetilde{x_1})}\PZ^1\times \PZ^1=T_{[1:0];[0:1]}\PZ^1\times \PZ^1$ which leads to the weights $-\varepsilon_1$ and $\varepsilon_1$. The last weight can be seen in the tangent space $T_{\widetilde{x_1}}(h^{-1}(E_1,E_3))=T_{[0:1]}\PZ(e_2,e_3)$. This leads to the weight $\varepsilon_2$. We summarise that the weights of $\widetilde{x_1}$ in $\widetilde{X}_4'$ are given by $-\varepsilon_1,\varepsilon_1$ and $\varepsilon_2$. The weights of the other $T$-fixed points in the exceptional divisor can be computed similarly. Therefore, for any $T$-fixed point $x\in X_5$ one can compute 
	\begin{align*}
		i_x^\ast[\widetilde{X}_5\to \IG(2,5)]=i_x^\ast[\widetilde{X}_4\to \IG(2,5)]+i_x^\ast[\widetilde{X}_4'\to \IG(2,5)].
	\end{align*}

	Lastly, we consider the singular subscheme $X_6\subseteq X$ which is given by 
	\begin{align*}
		X_6=\{V_2\subseteq \C^5\text{ isotropic}\ \big \vert \ V_2\cap \langle e_1,e_2,e_3\rangle\neq \emptyset\}.
	\end{align*}
	 We claim that a resolution $\widetilde{X}_6$ of $X_6$ is given by 
	\begin{align*}
		\widetilde{X}_6=\{(V_1,V_2,V_4)\in \PZ(\C^5)\times X_6\times \Gr(4,5)\ \vert \ V_1\subseteq V_2\cap E_3, V_4\supseteq V_2+E_3, V_4\subseteq V_1^\perp\}.
	\end{align*}
	Again, this is birational to $X_6$. Now, we want to show smoothness of $\widetilde{X}_6$. We consider the map
	\begin{align*}
		f:\widetilde{X}_6\to \{V_4\supseteq E_3\}=\PZ^1, (V_1,V_2,V_4)\mapsto V_4
	\end{align*}
	whose fiber is given by 
	\begin{align*}
		f^{-1}(V_4)=\{(V_1,V_2,V_4)\ \vert \ V_1\subseteq E_3, V_1\subseteq V_2\subseteq V_4, V_4\subseteq V_1^\perp\}
	\end{align*}
	where $V_4\subseteq V_1^\perp \Leftrightarrow V_1\subseteq V_4^\perp$ holds. Consider now the projection
	\begin{align*}
		g:f^{-1}(V_4)\to \{V_1\subseteq V_4^\perp\}\cong \PZ^1, (V_1,V_2,V_4)\mapsto V_1
	\end{align*}
	which is a $\PZ^2$-bundle over $\PZ^1$ because $V_4^\perp$ is two-dimensional. Thus, $f^{-1}(V_4)$ is smooth and therefore, $\widetilde{X}_6$ is smooth and projective.
	
	Now, we want to apply Proposition \ref{Non-smooth_classes} to obtain the pullback $i_x^\ast[\widetilde{X}_6\to \IG(2,5)]$ for the singular $T$-fixed points $x\in X_6$. The singular $T$-fixed points in $X_6$ are $x_{12},x_{13}$ and $x_{23}$. The $T$-fixed points in the exceptional divisor which map to $x_{12}$ are given by $(E_1,E_2,E_4)$ and $(e_2,E_2,\langle E_3,e_5\rangle)$. For the other two singular $T$-fixed points we obtain three $T$-fixed points in the exceptional divisor, e.g. $(e_1,\langle e_1,e_3\rangle,E_4), (e_3,\langle e_1,e_3\rangle,E_4)$ and $(e_3,\langle e_1,e_3\rangle,\langle E_3,e_5\rangle)$ are the $T$-fixed points in the fiber of $x_{13}$. Exemplary, we compute the weights for one of the $T$-fixed points in the fiber of $x_{12}$, i.e. $\tilde{x}:=(E_1,E_2,E_4)$. Therefore, we consider the morphism $f$ and the tangent space $T_{f(\tilde{x})}\PZ^1=T_{[1:0]}\PZ(e_4,e_5)$. Thus, we obtain the weight $-\varepsilon_1+\varepsilon_2$. Next, we need to compute the weights in $T_{\tilde{x}}(f^{-1}(E_4))$. Therefore, we consider the morphism $g$ and the tangent space $T_{g(\tilde{x})}\PZ^1=T_{[1:0]}\PZ(e_1,e_3)$ which leads to the weight $-\varepsilon_1$. Lastly, we consider the tangent space $T_{\tilde{x}}(g^{-1}(E_1))$ which are the two-dimensional spaces containing $e_1$ and contained in $E_4$. Thus, we obtain the last weights from $T_{\tilde{x}}(g^{-1}(E_1))=T_{[1:0:0]}\PZ(e_2,e_3,e_4)$. This leads to the weights $-\varepsilon_2,-2\varepsilon_2$. We summarise that the weights of $\tilde{x}$ in $\widetilde{X}_6$ are given by $-\varepsilon_1+\varepsilon_2,-\varepsilon_1,-\varepsilon_2$ and $-2\varepsilon_2$. Similarly, one can compute all the other weights and apply Proposition \ref{Non-smooth_classes} to finish the computation.     
\end{bei}
\begin{bem}
	Assuming we could determine the pullback at singular points using the equations given in Example \ref{ExampleIG(2,5)} and the weights acting on the tangent space at smooth points as in Chow rings (cf. \cite[Section 4]{BrionTorusActions}), we would be able to determine the class $[\widetilde{X}_4\to \IG(2,5)]$ uniquely. A computation shows that one cannot even determine a unique class in $K$-theory and in fact it is not even known whether these classes correspond to the resolutions of singularities of $X_4$. The fact that one cannot determine the class $[\widetilde{X}_4\to \IG(2,5)]$ uniquely is natural because two different resolutions of singularities determine two different classes in cobordism. For example, one could also consider another resolution of singularities of $X_4$ given by 
	\begin{align*}
		\widetilde{X}^\ast_4=\{(V_1,V_2)\in \PZ(\C^5)\times \IG(2,5)\ \vert \ V_1\subseteq \langle e_1,e_2\rangle, V_2\supseteq V_1 \text{ isotropic}\}
	\end{align*}
	which is a $\PZ^2$-fibration over $\PZ^1$. The exceptional locus of $\widetilde{X}^\ast_4$ over $X_4$ is a $\PZ^1$ over the singular point $x_{12}$. A computation shows that the classes $i_{x_{12}}^\ast[\widetilde{X}_4\to \IG(2,5)]$ and $i_{x_{12}}^\ast[\widetilde{X}^\ast_4\to \IG(2,5)]$ do not coincide, although they both reduce to same one in Chow rings. 
\end{bem}
	\bibliographystyle{acm}
	\bibliography{Bibfile}
\end{document}